\documentclass[12pt,a4paper]{amsart}
\usepackage{amsmath}
\usepackage{amsfonts}
\usepackage{amssymb}
\usepackage{url}
\def\Z{{\Bbb Z}}
\def\Q{{\Bbb Q}}    
\def\R{{\Bbb R}}
\def\C{{\Bbb C}}     
\def\F{{\Bbb F}}   
\def\cF{{\mathcal F}}
\def\p{{\partial}}
\def\fF{{\mathfrak F}}
\def\D{{\Bbb D}}
\def\a{{\alpha}}
\def\b{{\beta}}
\def\c{{\gamma}}
\def\d{{\delta}}
\def\fM{{\mathfrak M}}
\def\cH{{\mathcal H}}
\def\depth{{\rm depth}}
\def\X{{X}}
\def\cM{{\mathcal M}}
\def\F{{\mathcal F}}
\def\G{{\mathcal G}}
\def\Res{{\rm Res}}
\def\Sym{{\rm Sym}}

\newtheorem{theorem}{Theorem}[section]
\newtheorem{lemma}[theorem]{Lemma}
\newtheorem{proposition}[theorem]{Proposition}
\newtheorem{cor}[theorem]{Corollary}

\begin{document}
\title{Modules of Jacobi forms of degree two of small levels}
\author[Aoki and Ibukiyama]{Hiroki Aoki and Tomoyoshi Ibukiyama}
\address{Faculty of Science and Technology, 
Tokyo University of Science, Noda, Chiba, 278-8510}
\email{aoki\_hiroki\_math@nifty.com}
\address{Department of Mathematics, Graduate School of Science, The university of Osaka,
Machikaneyama 1-1, Toyonaka, Osaka, 560-0043 Japan}
\email{ibukiyam@math.sci.osaka-u.ac.jp}

\keywords{Jacobi forms, Siegel modular forms, Differential operators}
\subjclass[2020]{Primary: 11F46,11F50}
\thanks{The first author is supported by 
JSPS Kakenhi Grant Number JP23K03039. 
The second author is supported by 
JSPS Kakenhi Grant Number JP23K03031}

\maketitle

\begin{abstract}
The purpose of this paper is to describe explicitly 
the modules of (Siegel-)Jacobi forms of degree two of 
index one of any scalar valued weight with respect to 
some congruence subgroups of small levels $N\leq 4$.
Such a structure for the full Siegel modular group 
as a module over scalar valued Siegel modular forms of even weight 
has been explicitly given by T.~Ibukiyama.
There we used an explicit structure theorem of rings of scalar valued Siegel modular forms by Igusa and that of vector valued Siegel modular 
forms of weight $\det^k\, \Sym^2$ by T. Satoh and T. Ibukiyama.
On the other hand, for levels $N=2$, $3$, $4$, ring structures of scalar valued case 
 have been also known by H. Aoki and T. Ibukiyama and $\det^k \Sym^2$-valued case by H. Aoki.   
 In this paper, by merging these results, we give the same sort of simple 
 structure theorems 
 on modules of Jacobi forms of degree of two of index one for level $2$, $3$ and $4$. 
\end{abstract}

\section{Problems and theorems}
\allowdisplaybreaks[4]
The purpose of this paper is to describe explicitly 
the modules of (Siegel-)Jacobi forms of degree two of 
index one of any scalar valued weight with respect to 
some congruence subgroups of small levels $N\leq 4$.
In this section, we shortly explain results without going into details.
More explicit description will be explained in later sections.

Let $H_n$ be the Siegel upper half space of degree $n$
given by
\[
H_n=\{X+iY \in M_n(\C); X=\,^{t}X, Y=\,^{t}Y\in M_n(\R),  Y>0\}.
\] 
We denote by $Sp(n,\R)\subset M_{2n}(\R)$
the symplectic group acting on $H_n$ as usual.
We define a Siegel modular group by 
$\Gamma_n=Sp(n,\Z)=Sp(n,\R)\cap M_{2n}(\Z)$.
Let $\Gamma$ be a finite index subgroup of 
$Sp(2,\Z)$. 
A holomorphic function $\Phi(\tau,z)$ in $(\tau,z)\in H_2\times \C^2$
is said to be a Jacobi form $\Phi(\tau,z)$ of degree two of index $1$
of weight $k$ with respect to 
$\Gamma$, if it satisfies the following two conditions 
\eqref{jacobi1} and \eqref{jacobi2}
for any $g=\begin{pmatrix} A & B \\ C & D \end{pmatrix} \in \Gamma$
and row vectors $\lambda$, $\mu\in \Z^2$. 
Here we write $e(\alpha)=e^{2\pi i\alpha}$
for any $\alpha \in \C$.
\begin{align}
 \label{jacobi1}
 \Phi(g\tau,z(C\tau+D)^{-1})& =
 \det(C\tau+D)^{k}e(z(C\tau+D)^{-1}C\,^{t}z)\Phi(\tau,z),
 \\
\Phi(\tau,z+\lambda\tau+\mu)& = e(-\lambda\tau\,^{t}
\lambda-2\lambda\,^{t}z)
\Phi(\tau,z). \label{jacobi2}
\end{align} 
We 
denote the space of such Jacobi forms by $J_{k,1}(\Gamma)$.
For any $k\geq 0$ and even $j\geq 0$, 
we denote by $A_{k,j}(\Gamma)$ the space of 
Siegel modular forms of weight $\det^k \Sym^j$ with respect to $\Gamma$, where $\Sym^j$ is the symmetric tensor representation of $GL_2(\C)$ of degree $j$.
In this paper, we treat some special $\Gamma$ of small levels. 

For any positive integer $N$, we put 
\[
\Gamma_0^{(n)}(N)=\left\{\begin{pmatrix} A & B \\ C & D \end{pmatrix}\in Sp(n,\Z);
C\equiv 0 \bmod N\right\}.
\]
For $N=1$, we write $\Gamma_0^{(n)}(1)=Sp(n,\Z)$ or $\Gamma_n$. 
For $N=3$ and $4$, we define a character $\psi$ of $\Gamma_0^{(n)}(3)$, $\Gamma_0^{(n)}(4)$ 
by $\psi(g)=\left(\frac{-N}{\det(D)}\right)$ for $g=\begin{pmatrix} A & B \\ C & D 
\end{pmatrix}$,
where $\left(\dfrac{*}{*}\right)$ means the quadratic residue character
corresponding to $\Q(\sqrt{-N})$.
Since we mostly treat the case $n=2$, 
we simply write 
$\Gamma_0(N)=\Gamma_0^{(2)}(N)$. 
We also write 
\[
\Gamma_0^0(2)=\left\{
\begin{pmatrix} A & B \\ C & D \end{pmatrix}
\in Sp(2,\Z); B\equiv C\equiv 0 \bmod 2\right\}.
\]
The character $\psi$ can be similarly defined for 
$\Gamma_0^0(2)$.
If we put $\gamma_2=\begin{pmatrix} 1_2 & 0 \\ 0 & 2\cdot 1_2 \end{pmatrix}$, then $\gamma_2^{-1}\Gamma_0(4)\gamma_2=\Gamma_0^0(2)$, so it is clear that $A_{k,j}(\Gamma_0(4)) \cong A_{k,j}(\Gamma_0^0(2))$, 
but we will see later that 
$J_{k,1}(\Gamma_0(4))$ 
and $J_{k,1}(\Gamma_0^0(2))$ are quite different, since the
action of $\gamma_2$ changes the Heisenberg part and hence the index.  
For each $N=3$ or $4$, we write 
$\Gamma_0(N)^{\psi}=Ker(\psi)$, which is a subgroup of 
$\Gamma_0(N)$ of index two.
We also put $\Gamma_0^0(2)^{\psi}=Ker(\psi:\Gamma_0^0(2)\rightarrow \{\pm 1\})$. 
In this paper, we determine Jacobi forms of index $1$ with 
respect to $\Gamma=\Gamma_2$, $\Gamma_0(2)$, 
$\Gamma_0(3)^{\psi}$, $\Gamma_0(4)^{\psi}$ and 
$\Gamma_0^0(2)^{\psi}$. 
When $N=1$ or $2$, we put $\psi=1$ and sometimes write 
$\Gamma_0(N)^{\psi}=\Gamma_0(N)$ for $N=1$, $2$ to 
unify the notation.
For $\tau\in H_2$ and $z\in \C^2$, we write 
$\tau=\begin{pmatrix} \tau_{11} & \tau_{12} \\ \tau_{12} & \tau_{22}\end{pmatrix}$
and $z=(z_1,z_2)$. 
We put 
\[
I=\begin{pmatrix} 1 & 0 & 0 & 0 \\ 0 & -1 & 0 & 0 \\ 0 &0 & 1 & 0 \\ 0 & 0 & 0 & -1 
\end{pmatrix}.
\]
We assume for a while that $\Gamma$ is one of the above 
five groups. We have $\psi(I)=-1$ for $\Gamma_0(3)$, 
$\Gamma_0(4)$, $\Gamma_0^0(2)$, but 
since we have $I\Gamma=\Gamma I$, 
a space of elements invariant by an action of $\Gamma$ is also invariant by 
$I$. This means we have a direct sum decomposition  $A_k(\Gamma)=A_k^{+}(\Gamma)\oplus A_k^{-}(\Gamma)$
where for $\varepsilon=+$ or $-$ we put
\[
A^{\varepsilon}_{k}(\Gamma)=\left\{
F\in A_k(\Gamma); 
F\begin{pmatrix} \tau_{11} & -\tau_{12} \\
-\tau_{12} & \tau_{22}\end{pmatrix}
= \varepsilon 
F\begin{pmatrix} \tau_{11} & \tau_{12} \\
\tau_{12} & \tau_{22}\end{pmatrix}
\right\}.
\]
Similarly, we have $J_{k,1}(\Gamma)=J_{k,1}^{+}(\Gamma)\oplus J_{k,1}^{-}(\Gamma)$ 
where we put  
\[
J_{k,1}^{\varepsilon}(\Gamma)=
\left\{\Phi\in J_{k,1}(\Gamma); 
\Phi\left(\begin{pmatrix}\tau_{11} & -\tau_{12} \\ -\tau_{12} & \tau_{22} 
\end{pmatrix},\begin{pmatrix} z_1 \\ -z_2 \end{pmatrix}\right)
=\varepsilon \Phi(\tau,z)\right\}.
\]
For each $\varepsilon=\pm$, we put 
\begin{align*}
A^{\varepsilon}(\Gamma) =\oplus_{k=0}^{\infty}A_k^{\varepsilon}(\Gamma),
\qquad  
J_{*,1}^{\varepsilon}(\Gamma) =\oplus_{k=1}^{\infty}J_{k,1}^{\varepsilon}(\Gamma).
\end{align*}
Here $A^+(\Gamma)$ is a graded ring and 
$J_{*,1}^{\varepsilon}(\Gamma)$ is a graded module over $A^+(\Gamma)$.
We often call the case $\varepsilon=+$ as type I and the case $\varepsilon=-$ as 
type II, since the meanings of plus and minus are sometimes confusing.
So we write also $A^I(\Gamma)=A^+(\Gamma)$ and $A^{II}(\Gamma)=A^{-}(\Gamma)$,
or $J^{I}_{*,1}(\Gamma)=J^+_{*,1}(\Gamma)$ and $J^{II}_{*,1}(\Gamma)=J^{-}_{*,1}(\Gamma)$, 
replacing $\varepsilon$ by I and II according to $\varepsilon=+$ and $-$, respectively.

The purpose of this paper is to describe 
the modules $J_{*,1}^{I}(\Gamma)$ and  $J_{*,1}^{II}(\Gamma)$
concretely as modules over $A^I(\Gamma)$. 
For $N=1$ and $2$, we have 
$A^I(\Gamma)=\sum_{k=0}^{\infty}A_{2k}(\Gamma)$. 
For $N=3$ and $4$, if we put 
\[
A_k(\Gamma_0(N),\psi)=\left\{F\in A_k(\Gamma); F|_{k}[\gamma]=\psi(\gamma)F
\right\}.
\]
then it is ovbious that we have 
\[
A_k(\Gamma_0(N)^{\psi})=A_k(\Gamma_0(N))\oplus 
A_k(\Gamma_0(N),\psi)),
\]
and we also have 
\[
A^I(\Gamma_0(N))=\sum_{k:even}A_k(\Gamma_0(N)) \oplus 
\sum_{k:odd}A_k(\Gamma_0(N),\psi).
\]
By results in \cite{ibuaoki},  the rings $A^I(\Gamma)$ 
are all weighted polynomial rings for the above four 
$\Gamma$. They are generated by elements  
in $A_k^I(\Gamma) $ with 
$k \in \{4,6,10,12\}$ for $N=1$, $\{2,4,4,6\}$ for $N=2$, 
$\{1,3,3,4\}$ for $N=3$ and $\{1,2,2,3\}$ for $N=4$, respectively. 
 
\begin{theorem} 
Each space $J_{*,1}^{I}(\Gamma)$ or $J_{*,1}^{II}(\Gamma)$ in the left column 
of the following table 
is a free module over $A^+(\Gamma)$ of rank $4$
and the weights of generators are given in the right column.
\[ 
\begin{array}{|c|c|} \hline
space & weight \\ \hline
J_{*,1}^{I}(\Gamma_2) &  4,6,10,12 \\
J_{*,1}^{II}(\Gamma_2) &  21, 27, 29, 35 \\
J_{*,1}^{I}(\Gamma_0(2)) & 2, 4, 4, 6  \\
J_{*,1}^{II}(\Gamma_0(2)) & 13, 15, 17, 19 \\
J_{*,1}^{I}(\Gamma_0(3)^{\psi}) & 1, 3, 4, 6 \\
J_{*,1}^{II}(\Gamma_0(3)^{\psi}) & 9, 11, 12, 14 \\
J_{*,1}^{I}(\Gamma_0(4)^{\psi}) & 1, 2, 3, 4  \\
J_{*,1}^{II}(\Gamma_0(4)^{\psi}) & 7, 9, 11, 11 \\
J_{*,1}^{I}(\Gamma_0^0(2)^{\psi}) & 1, 2, 2, 3  \\
J_{*,1}^{II}(\Gamma_0^0(2)^{\psi}) & 9, 10, 10, 11 \\ \hline 
\end{array}
\]
All the generators are explicitly given, which will be explained 
later.
\end{theorem}
The result for $N=1$ has been already given in \cite{ibujacobi} 
and the point here is the case $N\geq 2$.

\begin{cor}
The generating functions of dimensions of Jacobi forms are given as follows.
\begin{align*}
\sum_{k=1}^{\infty}\dim J_{k,1}^I(\Gamma_2) t^k
& = 
\frac{t^4+t^6+t^{10}+t^{12}}
{(1-t^4)(1-t^6)(1-t^{10})(1-t^{12})}, \\
\sum_{k=1}^{\infty}\dim J_{k,1}^{II}(\Gamma_2)t^k 
& = 
\frac{t^{21}+t^{27}+t^{29}+t^{35}}
{(1-t^4)(1-t^6)(1-t^{10})(1-t^{12})}, \\
\sum_{k=1}^{\infty}\dim J_{k,1}^I(\Gamma_0(2))t^k 
& = 
\frac{t^2+2t^4+t^6}
{(1-t^2)(1-t^4)^2(1-t^6)}, \\
\sum_{k=1}^{\infty}\dim J_{k,1}^{II}(\Gamma_0(2))t^k 
& = 
\frac{t^{13}+t^{15}+t^{17}+t^{19}}
{(1-t^2)(1-t^4)^2(1-t^6)}, \\
\sum_{k=1}^{\infty}\dim J_{k,1}^I(\Gamma_0(3)^{\psi})t^k 
& = 
\frac{t+t^3+t^4+t^6}
{(1-t)(1-t^3)^2(1-t^4)}, \\
\sum_{k=1}^{\infty}\dim J_{k,1}^{II}(\Gamma_0(3)^{\psi})t^k 
& = 
\frac{t^9+t^{11}+t^{12}+t^{14}}
{(1-t^2)(1-t^3)^2(1-t^4)}, \\
\sum_{k=1}^{\infty}\dim J_{k,1}^I(\Gamma_0(4)^{\psi})t^k 
& = 
\frac{t+t^2+t^3+t^4}
{(1-t)(1-t^2)^2(1-t^3)}, \\
\sum_{k=1}^{\infty}\dim J_{k,1}^{II}(\Gamma_0(4)^{\psi})t^k 
& = 
\frac{t^7+t^9+t^{11}+t^{11}}
{(1-t)(1-t^2)^2(1-t^3)},\\
\sum_{k=1}^{\infty}\dim J_{k,1}^I(\Gamma_0^0(2)^{\psi})t^k 
& = 
\frac{t+2t^2+t^3}
{(1-t)(1-t^2)^2(1-t^3)}, \\
\sum_{k=1}^{\infty}\dim J_{k,1}^{II}(\Gamma_0^0(2)^{\psi})t^k 
& = 
\frac{t^9+2t^{10}+t^{11}}
{(1-t)(1-t^2)^2(1-t^3)}. \\
\end{align*}
\end{cor}

Our strategy of the proof of the above statements is shortly explained as follows.
For Jacobi forms of degree $n=1$, it is well known that the Taylor coefficients of Jacobi forms are essentially elliptic modular forms (see \cite{eichlerzagier}). We use the similar facts here for $n=2$. 
For any $\Phi(\tau,z) \in J_{k,1}(\Gamma)$, we consider the 
Taylor expansion of $\Phi$ around $z=(z_1,z_2)=0$. 
Since we see $\Phi(\tau,-z)=\Phi(\tau,z)$ by the action of $-1_4$
in \eqref{jacobi1}, we may write 
\begin{equation}\label{taylor}
\Phi(\tau,z)=f_0(\tau)+f_{20}(\tau)z_1^2+f_{11}(\tau)z_1z_2+f_{02}(\tau)z_2^2+\cdots             
\end{equation}
We denote by $A_{k,2}(\Gamma)$ the vector valued 
Siegel modular forms of weight $\det^k \Sym^2$ with respect to 
$\Gamma$, where $\Sym^2$ is the symmetric tensor representation 
of $GL_2(\C)$. The representation space of $\Sym^2$ is identified 
with $\C u_1^2+\C u_1u_2+ \C u_2^2$, where 
the action is given by $(u_1,u_2)\rightarrow (u_1,u_2)A$ for $A \in GL_2(\C)$.
Now we put  
\begin{align}
\xi_0(\Phi)  = & f_0(\tau),  \\
\xi_{k,2}(\Phi)  =  & f_{20}(\tau)u_1^2+f_{11}(\tau)u_1u_2+f_{02}(\tau)u_2^2 \label{taylordiff}
 \\ & -\frac{2\pi i}{k}\left(\frac{\p f_0}{\p \tau_{11}}u_1^2+\frac{\p f_0}{\p \tau_{12}}u_1u_2
+\frac{\p f_0}{\p \tau_{22}}u_2^2\right). \notag
\end{align}
We write $\xi(\Phi)=(\xi_0(\Phi), \xi_{k,2}(\Phi))$. 
Then by \cite{ibujacobi} Theorem 3.1, 
we have a linear mapping 
\[
J_{k,1}(\Gamma)\ni \Phi \rightarrow \xi(\Phi)\in A_k(\Gamma) 
\times A_{k,2}(\Gamma).
\]
The map $\xi$ is always injective (loc. cit.), but not surjective in general. 
So the main point of the proof is to determine the image of this map.
The image is characterized by some vanishing conditions on various 
$\Gamma$-orbits of the diagonal components (See Proposition \ref{propwittcondition}). 
To check this conditions, we need explicit structures of $A_k(\Gamma)$ in \cite{ibuaoki}
and $A_{k,2}(\Gamma)$ in \cite{aokivector} and also 
the restriction of derivatives of $A_k(\Gamma)$ on various $\Gamma$-orbits in 
the $Sp(2,\Z)$-orbits of diagonals in $H_2$. 
Many arguments essentially depend on differential operators 
that preserves automorphy under the restrictions of domains.
The paper is organized as follows.
In Section 2, we review modular forms and fundamental properties of 
differential operators on those. In section 3, we give a proof for level $2$.
In section 4, we give a proof for level $3$, In section 5, 6, 7, we give proofs for two level $4$ cases.

\section{Review on modular forms and differential operators}
\subsection{Siegel modular forms and Jacobi forms}
For any finite dimensional representation $(\rho,V)$ of $GL_n(\C)$, 
the group $Sp(n,\R)$ acts on $V$-valued holomorphic functions $F$ of $H_n$ by 
\[
(F(Z)|_{\rho}[g])=\rho(C\tau+D)^{-1}F(gZ), \qquad 
g=\begin{pmatrix} A & B \\ C & D \end{pmatrix}
\in Sp(n,\R).
\]
When $\rho=\det^k$, we write $F|_{\rho}[g]=F|_{k}[g]$.
When it is obvious from the context, we simply write $F|[g]=F|_k[g]$. 

For a finite index subgroup $\Gamma$ of $Sp(n,\Z)$ with $n\geq 2$, 
we denote by $A_{\rho}(\Gamma)$ the space of 
$V$-valued holomorphic functions $F$ with $F|_{\rho}[\gamma]=F$
for all $\gamma \in \Gamma$. 
When $n=2$, $\rho=\det^k \Sym^j$, we write 
$A_{\rho}(\Gamma)=A_{k,j}(\Gamma)$. 
For a holomorphic function $\Phi(\tau,z)$ 
of $H_2\times \C^2$, $g =\begin{pmatrix} A & B \\ C & D 
\end{pmatrix} \in Sp(2,\R)$, $X=([\lambda,\mu],\kappa)$
with $\lambda$, $\mu \in \R^2$ and $\kappa \in \R$, 
we write 
\begin{align}
& \label{jacobi3} (\Phi|_{k,1}[g])(\tau,z) 
 =\det(C\tau+D)^{-k}e(-z(C\tau+D)^{-1}C\,^{t}z))
F(g\tau, z(C\tau+D)^{-1}), 
\\
& \label{jacobi4} 
(\Phi|_{k,1}[X])(\tau,z)   = e(\lambda\tau\,^{t}\lambda+
2\lambda \,^{t}z+\kappa)\Phi(\tau,z+\lambda\tau+\mu). 
\end{align}
The space $J_{k,1}(\Gamma)$ of Jacobi forms of degree two of weight $k$ of index $1$ 
with respect to $\Gamma$ is defined to be the space of holomorphic functions  
$\Phi(\tau,z)$ of $(\tau,z)\in H_2\times\C^2$ such that 
$\Phi|_{k,1}[\gamma]=\Phi$ ($\gamma \in \Gamma$) and 
$\Phi|_{k,1}[X]=\Phi$ ($X=(\lambda,\mu,0)$ with $\lambda$, $\mu\in \Z^2$).
(By the Koecher principle in \cite{ziegler}, conditions at cusps are not needed.)

\subsection{Differential operators on Siegel modular forms and Jacobi forms}
Since various differential operators on Siegel modular forms 
play essential roles in this paper, we shortly review 
them from \cite{ibudiff, ibuvect}.
We fix a partition $n=n_1+n_2$ ($n_i\geq 1$) and 
a sequence of integers $\lambda=(\lambda_1,\ldots,\lambda_m)$ with 
$\lambda_i\geq \lambda_{i+1}\geq 0$ for 
$i=1$, \ldots, $m-1$.  
Such $\lambda$ is called a dominant integral weight.
We call the maximum index $i$ such that 
$\lambda_i\neq 0$ a depth of $\lambda$. 
When $n_1$, $n_2\geq \depth(\lambda)$, there exist 
irreducible representations $\rho_{n_i,\lambda}$ of $GL_{n_i}(\C)$ corresponding to 
$\lambda$ for $n_1$ and $n_2$.
We embed $(g_1,g_2)\in Sp(n_1,\R)\times Sp(n_2,\R)$ in $Sp(n,\R)$ by 
\[
(g_1,g_2)\rightarrow \iota(g_1,g_2)=
\begin{pmatrix} A_1 & 0 & B_1 & 0 \\ 0 & A_2 & 0 & B_2 \\ C_1 & 0 & D_1 & 0 \\
0 & C_2 & 0 & D_2 \end{pmatrix},
\quad g_i=\begin{pmatrix} A_i & B_i \\ C_i & D_i \end{pmatrix} \in Sp(n_i,\R).
\]
For $Z\in H_n$, we put 
\[
Z=\begin{pmatrix} Z_{11} & Z_{12} \\ ^{t}Z_{12} & Z_{22} \end{pmatrix}, \quad Z_{ii}\in H_{n_{i}}, \ Z_{12}
\in M_{n_1,n_2}(\C).
\]
For a fixed positive integer $k$, a partition $n=n_1+n_2$, and a dominant integral weight 
$\lambda$, we consider a holomorphic linear differential operator $\D$ with constant coefficients
such that for any holomorphic function $f$ of $H_n$ and any $g_i\in Sp(n_i,\R)$, 
we have 
\begin{equation}\label{diffrelation}
\Res_{Z_{12}=0}(\D(f|_{k}[\iota(g_1,g_2)]))=(\Res_{Z_{12}=0}(\D f))|_{\det^k\rho_{n_1,\lambda}}[g_1]
|_{\det^k\rho_{n_2,\lambda}}[g_2],
\end{equation}
where $\Res_{Z_{12}=0}$ means the restriction to the domain $Z_{12}=0$ and 
the second action operates on variables $Z_{11}$ and $Z_{22}$, respectively.
This means that if $f$ is a Siegel modular form of weight $k$, then $\Res_{Z_{12}=0}(\D f)$ is also a Siegel modular form of weight $\det^k \rho_{n_i,\lambda}$ for each 
variable $Z_{ii}$. 
Under the assumption $k\geq n_i$, such non-zero 
differential operator exists
uniquely up to constant.
For example, if $n_1=2$, $n_2=1$, $\lambda=(2,0)$, then 
$\rho_{2,\lambda}=\Sym^2$, $\rho_{1,\lambda}=\det^2$. 
In this case, an explicit shape of $D$ is written in 
\cite{ibudiff} p. 114, and we can also apply 
$\D$ on Jacobi forms of degree two as follows.
We embed $g\in Sp(2,\R)$ to $Sp(3,\R)$ as before and 
$X=[(\lambda,\mu),\kappa]
\in \R^4\times\R$ to $Sp(3,\R)$ as 
\[
\begin{pmatrix} 1_2 & 0 & 0 & ^{t}\mu \\
\lambda & 1 & \mu & \kappa \\
0 & 0 & 1 & -\,^{t}\lambda \\
0 & 0 & 0 & 1 \end{pmatrix}.
\]
We denote the set of all such $X$ by $\cH(\R)$
and call it the Heisenberg group.
The group generated by $\iota(g_1,1_2)$ 
with $g_1\in Sp(2,\R)$ and $\cH(\R)$ 
is defined to be a real Jacobi group $J(\R)$ of degree two
in $Sp(3,\R)$. 
We denote by $\cH(\Z)$ the subgroup of $\cH(\R)$ consisting of 
rational integer components. We call here the group $\Gamma^J$ 
generated by $\iota(\Gamma,1_2)$ and 
$\cH(\Z)$ as a Jacobi group of $\Gamma$.
For any function $\Phi(\tau,z)$ invariant by $\cH(\Z)$, 
and for any $\gamma \in Sp(2,\Z)$, the function $\Phi|_{k,1}[\gamma]$ is
also invariant by $\cH(\Z)$. Indeed, we have 
$g\cdot X=X_g\cdot g$ for $X_g=([\lambda,\mu]g^{-1},\kappa)$,
and if $\Phi|_{k,1}[X]=\Phi$, then 
we have $(\Phi|_{k,1}[g])|[X]=(\Phi|[X_g])|_{k,1}[g]=\Phi|_{k,1}[g]$. 

Now for a holomorphic function $\Phi(\tau,z)$ on $H_2\times \C^2$
and $w\in H_1$, we put $\fF(Z)=\Phi(\tau,z)e(\omega)$, 
For any element $\mathfrak{g}\in J(\R)\subset Sp(3,\R)$, 
we see that $\fF(Z)|_{k}[\mathfrak{g}]=\tilde{\Phi}(\tau,z)e(\omega)$ 
for some function $\tilde{\Phi}$ which does not depend on $\omega$.
We write $\tilde{\Phi}=\Phi|_{k,1}[\mathfrak{g}]$. 
For the differential operator $\D$ we defined above,
we have a differential operator $\D_J$ on holomorphic functions of $H_2\times \C^2$ 
such that  
\[
\Res_{Z_{12}=0}\D(\fF(Z))=e(\omega)\Res_{z=0}(\D_J\Phi(\tau,z)).
\]
By the defining property of $\D$, we have 
\[
\Res_{z=0}(\D_J(\Phi(\tau,z)|_{k,1}[g]))=
(\D_J\Phi)(\tau,0)|_k[g]
\]
for any $g \in Sp(2,\R)$. 
When $n=3$, $n_1=2$, $n_2=1$ and $\lambda=(2,0)$, 
the operator $\D_J$ is explicitly given (up to constant) by 
\begin{align}\label{jacobiddiff}
\D_J =& u_1^2
\left(\frac{1}{2}\frac{\p^2\Phi}{\p z_1^2}-\frac{2\pi i}{k}\frac{\p\Phi}{\p \tau_{11}}\right)
+u_1u_2\left(\frac{\p^2\Phi}{\p z_1\p z_2}-\frac{2\pi i}{k}\frac{\p \Phi}{\p \tau_{12}}\right)
\\
& +u_2^2\left(\frac{1}{2}\frac{\p^2\Phi}{\p z_2^2}
-\frac{2\pi i}{k}\frac{\p\Phi}{\p \tau_{22}}\right).
\notag
\end{align}
Defining $\xi_{k,2}$ by \eqref{taylordiff}, we have 
\[
\xi_{k,2}(\Phi)=\Res_{z=0}\D_{J}(\Phi).
\]
Later we use two more similar differential operators on functions of $H_2$ 
preserving automorphy under the restriction to $H_1\times H_1$.
One is $\p_{12}=\frac{1}{2\pi i}\frac{\p}{\p \tau_{12}}$ which sends 
weight $k$ to $(k+1,k+1)$ on $H_1\times H_1$ and another is 
\begin{equation}\label{EZdiff}
D_2=\frac{k}{2}\frac{\p^2}{\p \tau_{12}^2}-\frac{\p^2}{\p\tau_{11}\p\tau_{22}}
\end{equation}
which sends weight $k$ to $(k+2,k+2)$ on $H_1\times H_1$. 

We also introduce three differential operators of another type
(often called of Rankin-Cohen type).
These are operators which construct new modular forms starting with 
several given modular forms. The theoretical background is found in 
\cite{ibudiff} and the proofs are in 
\cite{ibuaoki, ibudiff, ibuvect, tsatoh}, so we do not repeat the 
details here.
Let $f_i(\tau)$ be functions on $H_2$ and $k_i$ be positive integers
for $i=1$, $2$, $3$, $4$.
For $\tau=(\tau_{ij})\in H_2$, we write 
\[
\p_{ij}=\frac{1}{2\pi i}\frac{\p}{\p \tau_{ij}}.
\]
We define 
\[
\{f_1,f_2\}=\sum_{1\leq i\leq j\leq 2}\left(k_1f_1\p_{ij}f_2-k_2f_2\p_{ij}f_1\right)u_{i}u_{j}.
\]
This is the operator defined in \cite{tsatoh}.
This operator has a property that  for any $g\in Sp(2,\R)$, we have 
\[
\{f_1|_{k_1}[g],f_2|_{k_2}[g]\} =\{f_1,f_2\}|_{\det^{k_1+k_2}\Sym^2}[g].
\]
In particular, if $f_i\in A_{k_i}(\Gamma)$ for any $\Gamma$, then 
$\{f_1,f_2\}\in A_{k_1+k_2,2}(\Gamma)$. 
The following operators were defined in \cite{ibuvect, ibuaoki}.
\begin{align}
\{f_1,f_2,f_3\}= & u_1^2\begin{vmatrix}
\p_{11}f_1 & \p_{11}f_2 & \p_{11}f_3 \\
\p_{12}f_1 & \p_{12}f_2 & \p_{12}f_3 \\
k_1f_1 & k_2f_2 & k_3f_3 
\end{vmatrix}
-2u_1u_2\begin{vmatrix} \p_{11}f_1 & \p_{11}f_2 & \p_{11}f_3 \\
k_1f_1 & k_2f_2 & k_3f_3 \\
\p_{22}f_1 & \p_{22}f_2 & \p_{22}f_3 \
\end{vmatrix}
\\ 
& +u_2^2\begin{vmatrix}
k_1f_1 & k_2f_2 & k_3f_3 \\
 \p_{12}f_1 & \p_{12}f_2 & \p_{12}f_3 \\
\p_{22}f_1 & \p_{22}f_2 & \p_{22}f_3 
\end{vmatrix},\notag
\\
\{f_1,f_2,f_3,f_4\} & =\begin{vmatrix} k_1f_1 & k_2f_2 & k_3f_3 & k_4f_4 \\
\p_{11}f_1 & \p_{11}f_2 & \p_{11}f_3 & \p_{11}f_4\\
\p_{12}f_1 & \p_{12}f_2 & \p_{12}f_3 & \p_{12}f_4 \\
\p_{22}f_1 & \p_{22}f_2 & \p_{22}f_3 & \p_{22}f_4
\end{vmatrix}.
\end{align}
For any $g\in Sp(2,\R)$, we have 
\begin{align*}
\{f_1|_{k}[g],f_2|_k[g],f_3|_k[g]\}& =\{f_1,f_2,f_3\}|_{k_1+k_2+k_3+1,2}[g], \\
\{f_1|_{k}[g],f_2|_k[g],f_3|_k[g],f_4[g]\}& = \{f_1,f_2,f_3,f_4\}|_{k_1+k_2+k_3+k_4+3}[g].
\end{align*}
In particular, if $f_i\in A_{k_i}(\Gamma)$ for $i=1$, $2$, $3$, $4$,
then we have $\{f_1,f_2,f_3\}\in A_{k_1+k_2+k_3+1,2}(\Gamma)$ and 
$\{f_1,f_2,f_3,f_4\}\in A_{k_1+k_2+k_3+k_4+3}(\Gamma)$. 

\subsection{Relation between Jacobi forms and Siegel modular forms}
For the readers' convenience, we 
review the property of the map $\xi(\Phi)$ for $\Phi\in J_{k,1}(\Gamma)$ 
from \cite{ibujacobi}.
For each $\nu=(\nu_1,\nu_2)\in (\Z/2)^2$, we define theta functions 
$\vartheta_{\nu}(\tau,z)$ on $H_2\times \C^2$ of the second kind by 
\[
\vartheta_{\nu}(\tau,z)=\sum_{p\in \Z^2}
e\left((p+\frac{\nu}{2}) \tau\,^{t}(p+\frac{\nu}{2})+2(p+\frac{\nu}{2})\,^{t}z\right).
\]
We have $\vartheta_{\nu}(\tau,-z)=\vartheta_{\nu}(\tau,z)$.
We write $\vartheta_{\nu}(\tau,0)=\vartheta_{\nu}(\tau)$.
It is well known that for any $\Phi \in J_{k,1}(\Gamma)$, 
there exists a holomorphic functions $c_{\nu}(\tau)$ such that 
\[
\Phi(\tau,z)=c_{00}(\tau)\vartheta_{00}(\tau,z)+
c_{01}\vartheta_{01}(\tau,z)+c_{10}(\tau)\vartheta_{10}(\tau,z)
+c_{11}(\tau)\vartheta_{11}(\tau,z).
\]
We put 
\begin{equation}\label{thetamat}
\Theta(\tau)=\begin{pmatrix}
\vartheta_{00}(\tau) & \vartheta_{01}(\tau) & \vartheta_{10}(\tau) & \vartheta_{11}(\tau) \\
\p_{11}\vartheta_{00}(\tau) & \p_{11}\vartheta_{01}(\tau) & \p_{11}\vartheta_{10}(\tau) & 
\p_{11}\vartheta_{11}(\tau) \\
\p_{12}\vartheta_{00}(\tau) & \p_{12}\vartheta_{01}(\tau) & \p_{12}\vartheta_{10}(\tau) & 
\p_{12}\vartheta_{11}(\tau) \\
\p_{22}\vartheta_{00}(\tau) & \p_{22}\vartheta_{01}(\tau) & \p_{22}\vartheta_{10}(\tau) & 
\p_{22}\vartheta_{11}(\tau) \\
\end{pmatrix}.
\end{equation}

We denote the Taylor expansion of $\Phi$ along $z=0$ as in 
\eqref{taylor}. Since 
\[
\frac{1+\delta_{ij}}{4(2\pi i)^2}\frac{\p^2\vartheta_{\nu}(\tau,z)}{\p z_i\p z_j}=
\frac{1}{2\pi i}\frac{\p \vartheta_{\nu}(\tau,z)}{\p \tau_{ij}},
\]
we have 
\begin{equation}\label{thetalin}
\Theta(\tau)\begin{pmatrix} c_{00}(\tau) \\ c_{01}(\tau) \\ c_{10}(\tau) \\ c_{11}(\tau)
\end{pmatrix}
=
\begin{pmatrix} f_0(\tau) \\ \frac{1}{2(2\pi i)^2}f_{20}(\tau) \\
\frac{1}{2(2\pi i)^2}f_{11}(\tau) \\ \frac{1}{2(2\pi i)^2}f_{02}(\tau)
\end{pmatrix}.
\end{equation}
Here $\det(\Theta(\tau))$ is not identically $0$.
So $c_{\nu}(\tau)$
 and hence $\Phi$ is determined uniquely by 
$f_0(\tau)$, $f_{ij}(\tau)$ and the map 
$\xi:J_{k,1}(\Gamma)\rightarrow A_k(\Gamma)\times A_{k,2}(\Gamma)$ is injective.
The problem is the image of this map.
We fisrt prove the following lemma.
\begin{lemma}\label{merojacobi}
For any $(f_0,h)\in A_k(\Gamma)\times A_{k,2}(\Gamma)$, we have 
a meromorphic Jacobi form $\Phi$ of weight $k$ of index $1$ with respect to 
$\Gamma$ such that $\xi(\Phi)=(f_0,h)$.
\end{lemma}

Here more precisely, a meromorphic Jacobi form $\Phi$ means that
there exist a meromorphic functions $c_{\nu}(\tau)$ such that 
\[
\Phi(\tau,z)=\sum_{\nu\in (\Z/2\Z)^2}c_{\nu}(\tau)\vartheta_{\nu}(\tau,z)
\]
and $\Phi|_{k,1}[\gamma]=\Phi$, $\Phi|[X]=\Phi$ for any $\gamma \in \Gamma$ 
and $X\in \cH(\Z)\cong (\Z)^4\cdot \Z$. 

\begin{proof}
For any $f_0\in A_k(\Gamma)$ and $h=h_{20}u_2^2+h_{11}u_1u_2+h_{02}u_2^2 
\in A_{k,2}(\Gamma)$, we define $f_{ij}(\tau)$ by 
\[
f_{ij}(\tau)=h_{ij}+\frac{2\pi i}{k}\frac{\p f_0}{\p\tau_{ij}}.
\]
Define $c_{\nu}(\tau)$ as a solution of the simultaneous equation 
\eqref{thetalin} and define 
\[
\Phi(\tau,z)=\sum_{\nu\in (\Z/2\Z)^2}c_{\nu}(\tau)\vartheta_{\nu}(\tau,z).
\]
Here obviously $c_{\nu}(\tau)$ are meromorphic functions such that 
$\det(\Theta)c_{\nu}(\tau)$ is holomorphic. 
For any $X \in \cH(\Z)$, we have $\vartheta_{\nu}(\tau,z)|[X]=
\vartheta_{\nu}(\tau,z)$ so we have $\Phi|[X]=\Phi$.
Now we prove $\Phi$ is invariant by the action of $\Gamma$.
We shortly write $\vartheta_{\nu}=\vartheta_{\nu}(\tau)$. 
For any $g\in Sp(2,\Z)$, we fix a branch of $\det(C\tau+D)^{1/2}$ and 
we define $f|_{1/2}[g]$ in the natural way for any function $f$ on $H_2$. 
Then it is well-known that 
there exists a constant matrix $R(g)$ such that 
\[
(\vartheta_{00}|_{1/2}[g],\vartheta_{01}|_{1/2}[g],\vartheta_{10}|_{1/2}[g],\vartheta_{11}|_{1/2}
[g])
=
(\vartheta_{00},\vartheta_{01},\vartheta_{10},\vartheta_{11})R(g).
\]
This means that there exists some linear combinations $c_{\nu}^{g}(\tau)$
of $c_{\mu}|_{k-1/2}[g]$  ($\mu \in (\Z/2\Z)^2$) 
such that 
\[
\Phi|_{k,1}[g]=\sum_{\nu\in (\Z/2\Z)^2}c_{\nu}^{g}(\tau)\vartheta_{\nu}(\tau,z).
\]
Now we would like to show that $\Phi|_{k,1}[\gamma]=\Phi$ for $\gamma 
\in \Gamma$. 
For any $g \in Sp(2,\R)$, we denote the first several Taylor coefficients of $\Phi|_{k,1}[g]$ along $z=0$ by $f_0^g$, $f_{ij}^g$ as 
\[
\Phi|_{k,1}[g]=f_0^g(\tau)+f_{20}^g(\tau)z_1^2+f_{11}^g(\tau)z_1z_2+
f_{02}^{g}(\tau)z_2^2+\cdots
\]
Then by definition we have 
\begin{align*}
\xi_0(\Phi|_{k,1}[g]) = & f_0^g, \\
\xi_2(\Phi|_{k,1}[g]) = & f_{20}^g(\tau)u_1^2+f_{11}^g(\tau)u_1u_2+f_{02}^g(\tau)u_2^2
\\ & -\frac{2\pi i}{k}\biggl(\frac{\p f_0^g}{\tau_{11}}u_1^2+\frac{\p f_0^g}{\p \tau_{12}}u_1u_2
+\frac{\p f_0^g}{\p \tau_{22}}u_2^2\biggr).
\end{align*}
By the property of the differential operators we used, 
we have $f_0^g=\xi_{0}(\Phi|_{k,1}[g])=\xi_0(\Phi)|_{k}[g]=f_0|_k[g]$ and 
$\xi_{2,k}(\Phi|_{k,1}[g])=(\xi_{k,2}\Phi)|_{k,2}[g]=h|_{k,2}[g]$ for any $g \in Sp(2,\R)$.
Since $f_0\in A_k(\Gamma)$ and $h\in A_{k,2}(\Gamma)$, 
we have $f_0|_{k}[\gamma]=f_0$ and $h|_{k,2}[\gamma]=h$ for any 
$\gamma \in \Gamma$. So we have $f_0^{\gamma}=f_0$, and 
\begin{align*}
h & = f_{20}(\tau)u_1^2+f_{11}(\tau)u_1u_2+f_{02}(\tau)u_2^2
-\frac{2\pi i}{k}\biggl(\frac{\p f_0}{\p\tau_{11}}u_1^2+\frac{\p f_0}{\p \tau_{12}}u_1u_2
+\frac{\p f_0}{\p \tau_{22}}u_2^2\biggr)
\\ & = 
f_{20}^{\gamma}(\tau)u_1^2+f_{11}^{\gamma}(\tau)u_1u_2+f_{02}^{\gamma}(\tau)u_2^2
\\ & -\frac{2\pi i}{k}\biggl(\frac{\p f_0^{\gamma}}{\p \tau_{11}}u_1^2+
\frac{\p f_0^{\gamma}}{\p \tau_{12}}u_1u_2
+\frac{\p f_0^{\gamma}}{\p \tau_{22}}u_2^2\biggr).
\end{align*} 
So we have $f_{ij}^{\gamma}(\tau)=f_{ij}(\tau)$. 
So $c_{\nu}(\tau)$ and $c_{\nu}^{\gamma}(\tau)$ are the solution of 
the same simultaneous equation \eqref{thetalin}, so we have 
$c_{\nu}^{\gamma}=c_{\nu}$.
So we have $\Phi|_{k,1}[\gamma]=\Phi$ for any $\gamma \in \Gamma$. 
\end{proof}

Next we consider conditions that $(f_0,h)\in A_k(\Gamma)\times A_{k,2}(\Gamma)$ 
is in the image of holomorphic Jacobi forms in $J_{k,1}(\Gamma)$. 
By Lemma \ref{merojacobi}, the only point is a condition that the above constructed 
Jacobi form is holomorphic. To see this we see $\Theta$ more carefully.
Let $\chi_5$ be the unique Siegel modular form of weight $5$ of level $1$ 
with multiplier whose Fourier expansion starts as 
\[
\chi_5=-e^{\pi i(\tau_{11}+\tau_{22}+\tau_{12})}+\cdots.
\]
It is well known that $\det(\Theta)=\chi_5/4$. 
Here in the standard fundamental domain $\cF$ of 
$Sp(2,\Z)$ in $H_2$, the Siegel modular form $\chi_5$ has zeros of order $1$ only on 
the diagonals (\cite{freitag}). 
 For any function $f(\tau)$ on $H_2$, we write 
\[
W f=f\begin{pmatrix} \tau_{11} & 0 \\ 0 & \tau_{22}\end{pmatrix}
\]
and call $W$ a Witt operator.
We put 
\[
\Gamma_0=\iota(SL_2(\Z)\times SL_2(\Z)).
\]
For $h\in A_{k,2}(\Gamma)$ and $g \in Sp(2,\R)$, we define 
$h_{ij}^g$ by 
\[
h|_{k,2}[g]=h_{20}^{g}u_1^2+h_{11}^{g}u_1u_2+h_{02}^gu_2^2.
\]
\begin{proposition}\label{propwittcondition}
(1) An element $(f_0,h)\in A_k(\Gamma)\times A_{k,2}(\Gamma)$ is in the image $\xi(J_{k,1}(\Gamma))$ if and only if 
\begin{equation}\label{wittcondition}
W\left(h_{11}^g+\frac{2\pi i}{k}\frac{\p(f_0|[g])}
{\p \tau_{12}}\right)=0
\end{equation}
for any $g\in Sp(2,\Z)$.
\\
(2) The condition (1) is satisfied if it is satisfied for all representatives $g$ of 
$\Gamma\backslash Sp(2,\Z)/\Gamma_0$
\end{proposition}

\begin{proof}
Since $\vartheta_{\nu}(\tau)$ are even function of $\tau_{12}$,  
we have $W(\p_{12}\vartheta_{\nu})=0$. So all the components of 
the third row of $\Theta$ vanishes under $W$.
So if $c_{\nu}^{g}(\tau)$ are holomorphic, then $W(f_{11}^g)=0$. This means  
that the condition \eqref{wittcondition} is necessary.
On the contrary, assume that \eqref{wittcondition} is satisfied for all $g\in Sp(2,\Z)$.
Then we have  
\[
\begin{pmatrix} c_{00}^{g} \\ c_{01}^g \\ c_{10}^g \\ c_{11}^g \end{pmatrix}
=
\Theta^{-1}
\begin{pmatrix} f_0^g \\ f_{20}^g/2(2\pi i)^2 \\ f_{11}^g/2(2\pi i)^2  \\
f_{02}^g /2(2\pi i)^2 
\end{pmatrix}.
\]
Since the first, second and fourth column of the cofactor matrix of $\Theta$ 
vanishes under $W$, 
all the components of $\det(\Theta)^{-1} \Theta$ 
except for the third column are holomorphic on $\cF$ . 
Since we assumed that $W(f_{11}^g)=0$, this means that 
$c_{\nu}^g(\tau)$ is also holomorphic on $\cF$. This means that 
$\Phi|_{k,1}[g]$ is holomorphic on $\cF\times \C^2$ for any $g \in Sp(2,\Z)$.
So $\Phi$ is holomorphic on $(\tau,z)\in g\cF\times \C^2$ for all $g\in Sp(2,\Z)$ and hence on the whole $H_2\times \C^2$. So we proved (1).
Now take $g \in Sp(2,\Z)$ written as $g=\gamma g_0 \gamma_0$ 
for some $\gamma\in \Gamma$,$\gamma_0\in \Gamma_0$.
We have $h|[\gamma]=h$ and $f_0|[\gamma]=f_0$. 
So $h^{g}=h|[g_0][\gamma_0]$ and $f^{g}=f_0|[g_0]|[\gamma_0]$.
We write $\gamma_0=\iota\biggl(\begin{pmatrix} a_1 & b_1 \\ c_1 & d_1 
\end{pmatrix},\begin{pmatrix} a_2 & b_2 \\ c_2 & d_2 \end{pmatrix}\biggr)$.
Then by definition we have
\[
(h|[g_0][\gamma_0])_{11}=(h|[g_0])_{11}(c_1\tau_{11}+d_1)^{k+1}(c_2\tau_{22}+d_2)^{k+1}.
\]
Since $\p_{12}$ is a differential operator well behaved to the restriction, we have 
\[
W(\frac{\p (f_0|[g_0])|[\gamma_0])}{\p \tau_{12}})
=
W(\frac{\p (f_0|[g_0])}{\p \tau_{12}})|[\gamma_0]
=
W(\frac{\p (f_0^{g_0})}{\p \tau_{12}})(c_1\tau_{11}+d_1)^{k+1}(c_2\tau_{22}+d_2)^{k+1}.
\]
So $(f_0,h)$ satisfies the condition in (1) for $g_0$, then it satisfies the 
condition also for $g$.
\end{proof}

\subsection{Witt operator and plus and minus part}
Now we will see more precisely how the plus and minus part 
of $J_{k,1}(\Gamma)$ behaves under 
the Witt operator.
We put 
\[
I=\begin{pmatrix} 1 & 0 & 0 & 0 \\ 
0 & -1 & 0 & 0 \\ 0 & 0 & 1 & 0 \\ 0 & 0 & 0 & -1 \end{pmatrix}
\]
as before. 
Then for $\tau\in H_2$, we have 
$I\tau=\begin{pmatrix} \tau_{11} & -\tau_{12} \\-\tau_{12}
& \tau_{22} \end{pmatrix}$.
We consider the case $\Gamma=\Gamma_0(N)^{\psi}$ for $N=1$, $2$, $3$, $4$ or $\Gamma_0^0(2)^{\psi}$. 
By definition, for $F\in A_k(\Gamma)$ and $\Phi\in J_{k,1}(\Gamma)$, 
we have $F|_k[I]=(-1)^kF(I\tau)$ and $\Phi|_{k,1}[I]=(-1)^k \Phi(I\tau,(z_1,-z_2))$. 
So for $\varepsilon=\pm$, we have 
\begin{align*}
A_k^{\varepsilon}(\Gamma)& =\{F\in A_k(\Gamma);(-1)^k (F|_{k}[I])=\varepsilon F\}, \\
J_{k,1}^{\varepsilon}&=\{\Phi\in J_{k,1}(\Gamma); (-1)^k\Phi|_{k,1}[I]=\varepsilon \Phi\}
\end{align*}
as defined before. 
If $f \in A_k^{+}(\Gamma)$, 
then $f$ is an even function with respect 
to $\tau_{12}$ so we have 
$W(\p_{12}f)=0$. 
In the same way, for $\varepsilon=\pm$, we define 
\[
A_{k,2}^{\varepsilon}(\Gamma)=\{h\in A_{k,2}(\Gamma);
(-1)^kh|_{k,2}[I]=\varepsilon h\}.
\]
If we write $h(\tau)=h_{20}(\tau)u_1^2+h_{11}(\tau)u_1u_2+h_{02}(\tau)u^2$, then 
$h\in A^{\varepsilon}_{k,2}(\Gamma)$ means that 
\[
h_{20}(I\tau)=\varepsilon h_{20}(\tau),
\qquad
h_{11}(I\tau)=(-1) \varepsilon h_{11}(\tau), 
\qquad 
h_{02}(\tau)=\varepsilon h_{02}(\tau).
\]
So for $h\in A_{k.2}^+(\Gamma)$, we have $W(h_{11})=0$.

Assume that $f_i \in A_{k_i}^{+}(\Gamma)$ 
for $i=1$, $2$, $3$, $4$. 
Then considering $f_i|_{k_i}[I]$ and using the property of the 
bracket that preserves automorphy, 
we can show 
\begin{align*}
\{f_1,f_2\}\in A_{k_1+k_2,2}^{+}(\Gamma) ,& 
\quad W(\{f_1,f_2\}_{11})=0, \\
\{f_1.f_2,f_3\} \in A_{k_1+k_2+k_3,2}^{-}(\Gamma), & \quad W(\{f_1,f_2,f_3\}_{11})=0, \\
\{f_1,f_2,f_3,f_4\} \in A_{k_1+k_2+k_3+k_4+3}^{-}(\Gamma),  
& \quad W(\{f_1,f_2,f_3,f_4\})=0. 
\end{align*}
Here for $h\in A_{k,2}(\Gamma)$, we write the coefficient of $u_1u_2$ in $h$ by $h_{11}$. 

By the way, note 
that we have the following Jacobi type identities, which will be 
often used.
\begin{align}
& k_1f_1\{f_2,f_3\}+k_2f_2\{f_3,f_1\}+k_3f_3\{f_1,f_2\}=0,
\label{rel2}
\\ & 
k_1f_1\{f_2,f_3,f_4\}+k_2f_2\{f_3,f_4,f_1\}+k_3f_3\{f_4,f_1,f_2\}
+k_4f_4\{f_1,f_2,f_3\}=0, \label{rel3}\\
& (\p_{12}f_1)\{f_2,f_3,f_4\}_{11}
-(\p_{12}f_2)\{f_3,f_4,f_1\}_{11}+
(\p_{12}f_3)\{f_4,f_1,f_2\}_{11} \label{ref4}
\\ & -(\p_{12}f_4)\{f_1,f_2,f_3\}_{11}
=2\{f_1,f_2,f_3,f_4\},\notag
\end{align}
where we write the coefficient 
of $u_1u_2$ of $h\in A_{k,2}(\Gamma)$  as $h_{11}$. 

For $\Phi\in J_{k,1}^{\varepsilon}(\Gamma)$, 
we have $\xi_0(\Phi)\in A_k^{\varepsilon}(\Gamma)$ and 
$\xi_{k,2}(\Phi) \in A_{k,2}^{\varepsilon}(\Gamma)$.

Now assume that $\xi(\Phi)=(g,h)$. 
Here we should note that for any 
$f\in A_{k_1}^{+}(\Gamma)$ and 
$\Phi\in J_{k_2,1}^{\varepsilon}(\Gamma)$, the vectors 
$\xi(f\Phi)$ do not coincide with 
$(fg,fh)$. So to make the product $f\Phi$ reflect, 
for a vector $(g,h)\in A_{k_2}(\Gamma)
\times A_{k_2,2}(\Gamma)$ and $f \in A_{k_1}(\Gamma)$, 
we define a product by  
\begin{equation}\label{moduleaction}
f(g,h):=\left(fg,fh+\frac{(2\pi i)^2}{k_2(k_1+k_2)}\{f,g\}\right).
\end{equation}
If $\xi(\Phi)=(g,h)$ for some $\Phi\in J_{k,1}(\Gamma)$, 
then we have $\xi(f\Phi)=f\xi(\Phi)$. 
In particular, we have 
\begin{align}
& \frac{k_1k_2}{(2\pi i)^2}(f(g,0)-g(f,0)) =(0,\{f,g\}),
\label{rel1}\\
& \frac{k_2}{k_1+k_2}f(g,0)+\frac{k_1}{k_1+k_2}
g(f,0)=(fg,0), \label{rel4} \\
& f(0,h)=(0,fh). \label{rel5}
\end{align}

\subsection{Theta constants and its transformation formula}
We sometimes use theta constants as generators of modular forms, 
in particular for level $2$ and $4$. We also need their derivatives to consider 
the condition of Proposition \ref{wittcondition} on $\Gamma$ orbits in 
the $Sp(2,\Z)$ orbits of diagonals of $H_2$, so here we gather necessary 
transformation formulas of theta constants extracted from 
\cite{igusagraded1, igusaderivative}.

For any integral vectors $m=(m',m'')\in \Z^{2n}$ and $\tau\in H_n$, $z\in \C^n$, 
we define 
\[
\theta_{m}(\tau,z)=\sum_{p\in \Z^n}e\biggl(\frac{1}{2}\,^{t}\bigl(p+\frac{m'}{2}\bigr)
\tau\bigl(p+\frac{m'}{2}\bigr)
+\,^{t}\bigl(p+\frac{m'}{2}\bigr)\bigl(z+\frac{m''}{2}\bigr)\biggr).
\]
In particular, we write $\theta_m(\tau)=\theta_m(\tau,0)$
and call it a theta constant. 
For $m=(m',m'')$,$n=(n',n'')\in \Z^{2n}$ we have 
\[
\theta_{m+2n}(\tau)=(-1)^{^{t}m'n''}\theta_m(\tau).
\]
In particular, $\theta_m^2(\tau)$ depends only on $m\bmod 2$. Or 
if $m'\equiv 0\bmod 2$, then $\theta_m(\tau)$ depends only on $m\bmod 2$.

For any $M=\begin{pmatrix} A & B \\ C & D \end{pmatrix}\in Sp(n,\Z)$ 
and $m\in \Z^{2n}$, we define
\[
M\circ m=\begin{pmatrix} D & -C \\ -B & A \end{pmatrix} m
+\begin{pmatrix} (C\,^{t}D)_0 \\ (A\,^{t}B)_0 \end{pmatrix},
\]
where for an $n\times n$ matrix $X$, we denote by $X_0$ the $n$-dimensional 
vector consisting of diagonal components of $X$. 
It is well known that this induces an action on $(\Z/2\Z)^{2n}$(\cite{igusagraded1}), that is, we have 
\[
(M_1M_2)\circ m\equiv M_1(M_2\circ m) \bmod 2.
\]
For $m=(m',m'')\in \Z^{2n}$ and $M\in Sp(n,\Z)$ as above,
we define 
\[
\phi_m(M)=-\frac{1}{8}\biggl(\mbox{}^{t}m'\,^{t}BDm'+\,^{t}m''\,^{t}ACm''-2\,^{t}m'\,^{t}BCm''
-2\,^{t}(A^{t}B)_0(Dm'-Cm'')\biggr).
\]
\begin{proposition}[Igusa \cite{igusagraded1} p.227] \label{transformation}
For $M\in Sp(n,\Z)$, we fix a branch of $\det(C\tau+D)^{1/2}$.
Then for any $m\in \Z^{2n}$, 
there exists an eighth root of unity $\kappa(M)$  depending only on $M$ and 
the branch of  $\det(C\tau+D)^{1/2}$ such that 
\[
\theta_{M\circ m}(M\tau)=\kappa(M)e(\phi_m(M))(C\tau+D)^{1/2}\theta_m(\tau).
\]
\end{proposition}
Here $\kappa(M)$ depends on a choice of branch of $\det(C\tau+D)^{1/2}$ but 
$\kappa(M)^2$ depends only on $M$. We can calculate $\kappa(M)^2$
as follows.
\begin{lemma}[Igusa \cite{igusagraded2} Lemma 4]\label{upperkappa}
For $M=\begin{pmatrix} A & B \\ 0 & D \end{pmatrix} \in Sp(n,\Z)$, or 
$\begin{pmatrix} A & 0 \\ C & D\end{pmatrix} \in Sp(n,\Z)$, we have 
\[
\kappa(M)^2=\det(D)=\det(A).
\]
\end{lemma}
Here of course $\det(A)=\det(D)=\pm 1$ since $\det(M)=1$
. 
\begin{lemma}\label{generatekappa}
(1) For any $M_1$, $M_2 \in Sp(n,\Z)$, we have 
\[
\kappa(M_1M_2)^2=\kappa(M_1)^2\kappa(M_2)^2e(2\phi_{M_2\circ 0}(M_1)).
\]
(2) $Sp(n,\Z)$ is generated by matrices of the forms  
$\begin{pmatrix} a & b \\ 0 & d \end{pmatrix}$ and 
$\begin{pmatrix} a & 0 \\ c & d \end{pmatrix}$.
\end{lemma}

\begin{proof} The second assertion is nothing but Lemma 3 of \cite{igusagraded2}. 
The first assertion is given as follows.
For $M=\begin{pmatrix} A & B \\ C & D \end{pmatrix}$, we write 
$J(M,\tau)=\det(C\tau+D)$. 
By Proposition \ref{transformation}, we have 
\begin{align*}
\theta_{(M_1M_2)\circ 0}^2(M_1M_2\tau)& =\kappa(M_1M_2)^2J(M_1M_2,\tau),
\\ 
\theta_0^2(\tau)
\theta_{M_1\circ m}^2(M_1\tau) & = e(2\phi_m(M_1))\kappa(M_1)^2J(M_1,\tau)\theta_m^2(\tau).
\end{align*}
We put $m=M_2\circ 0$ and replace $\tau$ by $M_2\tau$ in the second equality,
Then we have  $\theta_{M_2\circ 0}^2(M_2\tau)=\kappa(M_2)^2J(M_2,\tau)\theta_0^2(\tau)$ and 
\begin{multline*}
\theta_{M_1\circ (M_2\circ 0)}^2(M_1M_2\tau) 
\\ =\kappa(M_1)^2\kappa(M_2)^2e(2\phi_{M_2\circ 0}(M_1))J(M_1,M_2\tau)J(M_2,\tau)\theta_0^2(\tau).
\end{multline*}
Since $J(M_1,M_2\tau)J(M_2,\tau)=J(M_1M_2,\tau)$, comparing 
with the first equality, we have $\kappa(M_1M_2)^2=\kappa(M_1)^2\kappa(M_2)^2e(2\phi_{M_2\circ 0}(M_1))$.
\end{proof}

The following lemma is called the Jacobi's derivative formula.
\begin{lemma}[Igusa \cite{igusaderivative}]
\begin{equation}\label{jacobiderivative}
\frac{\p\theta_{11}}{\p z}(\tau,0)=-\pi \theta_{00}(\tau)\theta_{01}(\tau)\theta_{10}(\tau).
\end{equation}
\end{lemma}
Here we have 
\begin{equation}\label{der1}
\frac{1}{2\pi i}\frac{\p\theta_{11}}{\p z}(\tau,0)=
\sum_{p\in \Z}\bigl(p+\frac{1}{2}\bigr)e\biggl(\bigl(p+\frac{1}{2}\bigr)^2+
\bigl(p+\frac{1}{2}\bigr)\frac{1}{2}\biggr).
\end{equation}
For $\tau=\begin{pmatrix}\tau_{11} & \tau_{12} \\ \tau_{12} & \tau_{22} 
\end{pmatrix}\in H_2$, 
we write 
\[
\p_{ij}=\frac{1}{2\pi i}\frac{\p}{\p \tau_{ij}}
\]
as before.

Taking the derivative directly from the definition,  we see 
\[
W(\p_{12}\theta_{1111})=\prod_{i=1}^{2}\sum_{p\in \Z}\bigl(p+\frac{1}{2}\bigr)e\biggl(
\bigl(p+\frac{1}{2}\bigr)^{2}\tau_{ii}+\frac{1}{2}\bigl(p+\frac{1}{2}\bigr)\biggr).
\]
Comparing this with \eqref{der1} and \eqref{jacobiderivative},
we have 
\begin{equation}\label{1111derivative}
W(\p_{12}\theta_{1111})
=-\frac{1}{4}\prod_{i=1}^{2}\theta_{00}(\tau_{ii})\theta_{01}(\tau_{ii})\theta_{10}(\tau_{ii}).
\end{equation}
This relation will be used later.

\section{Proof for level $1$ and $2$}\label{prooflevel1and2}
We have $I\in \Gamma$ for $\Gamma=Sp(2,\Z)$, $\Gamma_0(2)$, so 
for $f \in A_k(\Gamma)$, we have $f\in A_k^+(\Gamma)$ 
when $k$ is even and $f \in A_k^{-}(\Gamma)$ when 
$k$ is odd. 
Since the result for $Sp(2,\Z)$ has been written in \cite{ibujacobi}, we explain mainly 
the case $\Gamma_0(2)$ and sketch the case $Sp(2,\Z)$. 
We have $A^+(\Gamma_2)=\C[\varphi_4,\varphi_6,\chi_{10},\chi_{12}]$ and 
$A^{-}(\Gamma_2)=\chi_{35}A^+(\Gamma_2)$ where 
$\varphi_k$ is the Siegel Eisenstein series of weight $k$ whose constant term is $1$ and
$\chi_{10}$, $\chi_{12}$, $\chi_{35}$ are the unique cusp form of weight $10$, $12$, $35$
normalized as in \cite{ibuaoki}.
We have $A^+(\Gamma_0(2))=\C[X_2,Y_4,Z_4,K_6]$ and $A^-(\Gamma_0(2))=
\chi_{19}A^+(\Gamma_0(2))$ where
\begin{align*}
X_2 = & (\theta_{0000}^4+\theta_{0001}^4+\theta_{0010}^4+\theta_{0011}^4)/4, \\
Y_4= & (\theta_{0000}\theta_{0001}\theta_{0010}\theta_{0011})^2, \\
Z_4 = & (\theta_{0100}^4-\theta_{0110}^4)^2/16384, \\
K_6 = & (\theta_{0100}\theta_{0110}\theta_{1000}\theta_{1001}\theta_{1100}\theta_{1111})^2
/4096, \\
\chi_{19}= & \{X_2,Y_4,Z_4,K_6\}/512.
\end{align*}
These are Siegel modular forms of $\Gamma_0(2)$ of weight $2$, $4$, $4$, $6$, $19$ 
defined in \cite{ibuthesis, ibulevel3, ibuaoki}.
For any positive integer $N$, we 
define a principal congruence subgroup $\Gamma(N)$ by 
\[
\Gamma(N)=\{\gamma\in Sp(2,\Z);\gamma \equiv 1_4 
\bmod N\}.
\]
This is a normal subgroup of $Sp(2,\Z)$. 
We have $I\in \Gamma(2)$, 
so we have $gIg^{-1}\in \Gamma(2) \subset \Gamma_0(2)$
for any $g \in Sp(2,\Z)$. 
This means that for any $f_0\in A_k(\Gamma)$
, $h\in A_{k,2}(\Gamma)$  for $\Gamma=
\Gamma_0(2)$ or $\Gamma_2$,  
and for any $g\in Sp(2,\Z)$, 
we have $f_0|_k[g]|_{k}[I]=f_0|_{k}[g]$ and 
$h|_{k,2}[g]|_k[I]=h_{k,2}|_{k}[g]$. So if $f_0\in A_k^+(\Gamma)$
 (in this case the sign $+$ means that $k$ is even), 
then $f_0|_k[g]$ is an even function with respect to $\tau_{12}$ and hence we have 
\[
W\left(\frac{\p (f_0|_k[g])}{\p \tau_{12}}\right)=0.
\]
This means that for any $f_0 \in A_k^+(\Gamma)$, 
there exists a Jacobi form $\Phi\in J_{k,2}^{+}(\Gamma)$ 
such that 
\[
\xi(\Phi)=(\xi_0(\Phi),\xi_{k,2}(\Phi))=(f_0,0)\in A_k(\Gamma)
\times A_{k,2}(\Gamma).
\]
First we consider the case $\Gamma=\Gamma_0(2)$.
(The case $\Gamma=\Gamma_2$ will be explained later.) 
We write the Jacobi form corresponding 
 to $(X_2,0)$, $(Y_4,0)$, $(Z_4,0)$, $(K_6,0)$ by 
 $\psi_2$, $\psi_{4,1}$, $\psi_{4,2}$, $\psi_{6}$, respectively.
For any $F\in A_k^+(\Gamma_0(2))$, we have 
$f_i\in A^+(\Gamma_0(2))$ such that 
$F=f_1X_2+f_2Y_4+f_3Z_4+f_4K_6$, so if we put 
$\Phi_0=f_1\psi_2+f_2\psi_{4,1}+f_3\psi_{4,2}+f_4\psi_{6}\in J_{k,1}^+(\Gamma_0(2))$, 
then $\xi_0(\Phi_0)=F$. So for any $\Phi\in J_{k,1}^{+}(\Gamma_0(2))$, 
there exists $\Phi_0\in J_{k,1}^+(\Gamma_0(2))$ such that 
$\xi_0(\Phi-\Phi_0)=0$. Now 
we will determine $\Phi\in J_{k,1}^+(\Gamma_0(2))$ such that $\xi(\Phi)=(0,h)\in A_k(\Gamma_0(2))\times A_{k,2}^+(\Gamma_0(2))$. 
As we have seen, for any $F_1\in A_{k_1}^+(\Gamma)$, $F_2\in A_{k_2}^+(\Gamma)$, 
there exist $\Phi_i\in J_{k,1}^+$ such that 
$\xi(\Phi_i)=(F_i,0)$.
By \eqref{rel1}, we have 
\begin{align*}
\xi_0(F_1\Phi_2-F_2\Phi_1) & = 0, \\
\xi_{k_1+k_2,2}(F_1\Phi_2-F_2\Phi_1) & = \frac{(2\pi i)^2}{k_1k_2}\{F_1,F_2\}.
\end{align*}
Since $A_{*,2}^+(\Gamma_0(2)):=\oplus_{k=0}^{\infty}A_{2k,2}(\Gamma_0(2))$ is generated by $\{f_1,f_2\}$ for $f_1$, $f_2\in 
A^+(\Gamma_0(2))$ by \cite{aokivector}, we see that 
the map 
$\xi:J_{k,1}^+(\Gamma_0(2))\rightarrow A_k^+(\Gamma)\times A_{k,2}^+(\Gamma)$ is 
surjective. except for $k=0$. So we have 
\[
J_{k,1}^+(\Gamma_0(2)) =A^+(\Gamma_0(2))\psi_2+A^+(\Gamma_0(2))\psi_{4,1}
+A^+(\Gamma_0(2))\psi_{4,2}+A^+(\Gamma_0(2))\psi_{6}.
\]
Since the map $\xi$ is injective, we have 
\[
\dim J_{k,1}^+(\Gamma_0(2))=\dim A_k^+(\Gamma_{0}(2)) 
\oplus \dim A_{k,2}^+(\Gamma_0(2)).
\]
By \cite{ibuaoki, aokivector}, we have 
\begin{align*}
\sum_{k=0}^{\infty}\dim A_k^+(\Gamma_0(2))t^k=& \frac{1}{(1-t^2)(1-t^4)^2(1-t^6)}, \\
\sum_{k=1}^{\infty}\dim A_{k,2}^+(\Gamma_0(2))t^k=& 
\frac{2t^6+2t^8+t^{10}-2t^{12}-t^{14}+t^{16}}{(1-t^2)(1-t^4)^2(1-t^6)}.
\end{align*}
so we have 
\begin{align*}
\sum_{k=1}^{\infty}\dim J_{k,1}^+(\Gamma_0(2))t^k = &  
\sum_{k=1}(\dim A_k^+(\Gamma_0(2))+
\dim A_{k,2}^+(\Gamma_0(2)) 
\\ = & 
\frac{t^2+2t^4+t^6}{(1-t^2)(1-t^4)^2(1-t^6)}.
\end{align*}
This dimension formula means that $\psi_{2}$, $\psi_{4,1}$, $\psi_{4,2}$, $\psi_6$ are
free over $A^+(\Gamma_0(2))$. 

Next we consider $J_{k,1}^{-}(\Gamma_0(2))$. 
We first show that there exists 
$\Phi\in J_{k,1}^-(\Gamma_0(2))$ such that 
$\xi_0(\Phi_0)=\chi_{19}$. This is not obvious as before since  
$W\left(\frac{\p (\chi_{19}|[g])}{\p \tau_{12}})\right)\neq 0$ for any $g \in Sp(2,\Z)$.
The point is to find a vector valued modular form $h\in A_{k,2}^{-}(\Gamma_0(2))$ such that 
\[
W\left(\frac{2\pi i}{19}\frac{\p (\chi_{19}|[g])}{\p \tau_{12}}+h_{11}^g\right)=0
 \]
 for any $g\in Sp(2,\Z)$. 
We first see the case when $g=1_4$. 
By $W(\theta_{1111})=0$, we have $W(K_6)=0$. 
Since $K_6$ is an even function with respect to $\tau_{12}$, we also have 
$W(\p_{ij}K_6)=0$ for any $i$, $j$. 
We also have $W(\p_{12}X_2)=W(\p_{12}Y_4)=W(\p_{12}Z_4)=0$. The 
derivative $\p_{12}\chi_{19}$ is the summation of the 
determinants replacing the each row of the 
defining determinant of $\chi_{19}$ 
by derivatives by $\p_{12}$.
Under $W$, the derivatives of each row except for the third row vanishes. The fourth column vanishes under $W$, so expanding the third row after making $\p_{12}$
operate, we have   
\[
512W(\p_{12}\chi_{19})=-W(\p_{12}^2K_6)W(\{X_2,Y_4,Z_4\}_{11})/2.
\]
For the sake of simplicity, we write 
\begin{align*}
& A=\theta_{00}(\tau_{11}), \quad 
B=\theta_{01}(\tau_{11}),  \quad 
C=\theta_{10}(\tau_{11}), \quad 
\\ & a=\theta_{00}(\tau_{22}), \quad  
b=\theta_{01}(\tau_{22}),  \quad 
c=\theta_{10}(\tau_{22}),
\end{align*}
throughout the paper. It is well known that 
$A^4=B^4+C^4$ and $a^4=b^4+c^4$ (\cite{igusagraded1}).
Using  $W(\theta_{1111})=W(\p_{12}^2\theta_{1111})=0$ and  
\eqref{1111derivative}, we have 
\[
W(\p_{12}^2K_6)=2A^2B^2C^6a^2b^2c^6\times W((\p_{12}\theta_{1111})^2)/4096
=A^4B^4C^8a^4b^4c^8/32768.
\]
On the other hand, we have 
\[
W(Y_4Z_4)=A^4B^4C^8a^4b^4c^8/16384.
\]
So we have 
\[
W(\p_{12}^2K_6)=W(Y_4Z_4)/2.
\]
This means that 
\[
W(512\p_{12}\chi_{19}+Y_4Z_4\{X_2,Y_4,Z_4\}_{11}/4)=0.
\]
 To see the image of $W$ in the condition \eqref{wittcondition}
for $g\neq 1_4$, we must consider the behaviour on 
 the orbits $g\begin{pmatrix} \tau_{11} & 0 \\ 0 & \tau_{22} \end{pmatrix}$.
 By Proposition \ref{wittcondition} (2), we need 
 representatives of $\Gamma_0(2)\backslash Sp(2,\Z)/\Gamma_0$
 for $\Gamma_0=\iota(SL_2(\Z)\times SL_2(\Z))$. 

\begin{lemma}\label{rightrep}
For any prime $p$, a complete set of representatives of 
$\Gamma_0(p)\backslash Sp(2,\Z)$ is given by 
\[
\begin{pmatrix} 0 & 0 & 1 & 0 \\ 0 & 0 & 0 & 1 \\
-1 & 0 & a & b \\ 0 & -1 & b & c \end{pmatrix},
\begin{pmatrix}
0  & 0 & 1 & 0 \\ 0 & 1 & 0 & 0 \\
-1 & b & a & 0 \\ 0 & 0 & b & 1 \end{pmatrix},
\begin{pmatrix} 1 & 0 & 0 & 0 \\
0 & 0 & 0 & 1 \\ 0 & 0 & 1 & 0 \\
0 & -1 & 0 & a \end{pmatrix}, 
\begin{pmatrix} 1 & 0 & 0 & 0 \\
0 & 1 & 0 & 0 \\ 0 & 0 & 1 & 0 \\ 0 & 0 & 0 & 1 \end{pmatrix},
\]
where $a$, $b$, $c$ runs over representatives of $p\Z$.
\end{lemma}

\begin{proof} This is obtained by taking the inverse of 
representatives of $Sp(2,\Z)/\Gamma_0(p)$ given in 
\cite{hashimoto} p.470 (7.4).
\end{proof} 

\begin{lemma} \label{level2rep}
A complete set of representatives of $\Gamma_0(2)\backslash Sp(2,\Z)/\Gamma_0$ 
is given by $\{1_4,M_1\}$ where 
\[
M_1=\begin{pmatrix} 1_2 & 0  \\
S_0 & 1_2 \end{pmatrix}, \qquad S_0=\begin{pmatrix} 0 & 1 \\ 1 & 0 \end{pmatrix}.
\]
\end{lemma}
\begin{proof}
 Among representatives in Lemma \ref{rightrep}, 
 there are only $6$ elements that are not contained in $\Gamma_0$, 
 so we check those directly and we see that 
 $1_4$ and $\begin{pmatrix} 0 & 1_2 \\ -1_2 & S_0 \end{pmatrix}$ are representatives.
By multiplying $\begin{pmatrix} 0 & -1_2 \\ 1_2 & 0 \end{pmatrix} \in \Gamma_0$ 
to the latter from right, we have $M_1$. \end{proof}

Since the bracket 
preserves the action of $Sp(2,\R)$, we have 
\[
\chi_{19}|[g]=\{X_2|_2[g],Y_4|_4[g],Z_4|_4[g],K_6|_6[g]\}
\]
for any $g \in Sp(2,\R)$. When $g=M_0$, we can write down this action by the transformation formula of 
theta constants.
By Lemma \ref{upperkappa}, we have $\kappa(M)^2=1$. 
If $m=M\circ n$, then by Proposition \ref{transformation}, we have 
$\theta_{m}(M\tau)(C\tau+D)^{-1/2}=\kappa(M)e(\phi_n(M))\theta_{n}(\tau)$.
By this we have 
\begin{align*}
X_2|[M_1]= & (\theta_{0000}^4+\theta_{0110}^4+\theta_{1001}^4+\theta_{1111}^4)/4, \\
Y_4|[M_1]= & -(\theta_{0000}\theta_{0110}\theta_{1001}\theta_{1111})^2, \\
Z_4|[M_1]= & (\theta_{0110}^4-\theta_{0010}^4)^2/16384, \\
K_6|[M_1] = & -(\theta_{0100}\theta_{0010}\theta_{1000}\theta_{0001}
\theta_{1100}\theta_{0011})^2/4096.
\end{align*}
So we have $W(Y_4|[M_1])=0$ and since $Y_4|[M_1]$ 
is an even function with respect to $\tau_{12}$. we have $W(\p_{ij}Y_4|_4[M_1])=0$ 
for any $i$, $j$. 
So we have 
\[
512W(\p_{12}(\chi_{19}|_{19}[M_1])
=
-W(\p_{12}^2(Y_4|_4[M_1]))W(\{X_2|[M_1],Z_4|[M_1],K_6|[M_1]\}_{11})/2.
\]
Then by \eqref{1111derivative}, we have 
\begin{align*}
W(\p_{12}^2(Y_4|[M_1]))& =-2A^2B^2C^2a^2b^2c^2W((\p_{12}\theta_{1111})^2)
\\ & = -(1/8)A^4B^4c^4a^4b^4c^4,
\\
W(K_6|[M_1])& = -A^4B^4C^4a^4b^4c^4/4096.
\end{align*}
So we have 
\[
W(\p_{12}^2(Y_4|[M_1])-512W(K_6|[M_1])=0.
\]
If we put 
$h=c_1h^{(1)}+c_2h^{(2)}$
for $h^{(1)}=Y_4Z_4\{X_2,Y_4,Z_4\}$ and $h^{(2)}=K_6\{X_2,Z_4,K_6\}$ for some constants $c_1$, $c_2$, 
then $h\in A_{19,2}(\Gamma_0(2))$ and 
$W(h_{11})=c_1W((h^{(1)})_{11})$, $W(h^{M_0}_{11})=c_2
W((h^{(2)})^{M_0}_{11})$. 
If we put 
$c_1=\frac{(2\pi i)^2}{19}\frac{1}{256}$ and $c_2=\frac{(2\pi i)^2}{19}$,
then we have 
\[
W\biggl(\frac{2\pi i}{19}\frac{\p(\chi_{19}|[M_1])}{\p \tau_{12}}+h^{M_1}_{11}\biggr)=0.
\]
This means that there exists $\Phi_{19}\in J_{k,1}(\Gamma_0(2))$ 
such that 
\[
\xi(\Phi_{19})=(\chi_{19},h).
\]
Now if we assume that $\Phi\in J_{k,1}^-(\Gamma_0(2))$, 
then since $A^-(\Gamma)=\chi_{19}A^+(\Gamma_0(2))$, there exists 
$f\in A^+(\Gamma_0(2))$ such that $\xi_0(\Phi-f\phi_{19})=0$.  
So it is enough to consider $\Phi \in J_{k,1}(\Gamma_0(2))$ such that 
$\xi_0(\Phi)=0$. Then the condition for existence of 
$\Phi$ for $h\in A_{k,2}(\Gamma_0(2))$ such that $\xi_2(\Phi)=(0,h)$ is 
$W(h_{11})=0$ and $W(h_{11}^{[M_1]})=0$. 
It is clear that $W(h_{11})=W((h^{[M_1]})_{11})=0$ for 
$h=\{X_2,Y_4,K_6\}$, $\{Y_4,Z_4,K_6\}$ since $W(Y_4|[M_1])=W(K_6)=0$. 
We consider $h=f_1\{X_2,Z_4,K_6\}+f_2\{X_2,Y_4,Z_4\}$ with $f_1\in A_{k-13}^+(\Gamma_0(2))$, 
$f_{2}\in A_{k-11}^{+}(\Gamma_0(2))$ and see the condition that 
$W(h_{11})=W((h^{M_1})_{11})=0$. Since we see $W(\{X_2,Y_4,Z_4\}_{11})\neq 0$
(e.g. by calculating the Fourier coefficients), 
we have $W(f_2)=0$, which means $f_2\in K A^+(\Gamma_0(2))$. 
Taking $h|[M_1]$, we also see that $W(f_1|[M_1])=0$.
This means that $f_1\in Y A^+(\Gamma_0(2))$. 
Since we have
\[
k_1f_1\{f_2,f_3,f_4\}+k_2f_2\{f_3,f_4,f_1\}+k_3f_3\{f_4,f_1,f_2\}+k_4f_4\{f_1,f_2,f_3\}=0
\]
for any $f_i\in A_{k_i}(\Gamma)$, we have 
\[
Y_4\{X_2,Z_4,K_6\}\in A^+(\Gamma)K_6\{X_2,Y_4,Z_4\}+A^+(\Gamma)\{X_2,Y_4,K_6\}+
A^+(\Gamma)\{Y_4,Z_4,K_6\}.
\]
So the image of $\Phi\in J_{k,1}(\Gamma_0(2))$ such that $\xi_0(\Phi)=0$ is 
\[
 A^+(\Gamma_0(2))K\{X_2,Y_4,Z_4\}+A^+(\Gamma_0(2))\{X_2,Y_4,K_6\}+
A^+(\Gamma_0(2))\{Y_4,Z_4,K_6\}.
\]
So $J_{k,1}^-(\Gamma_0(2))$ is spanned by 
the inverse images by $\xi$ of $(\chi_{19},h)$, $(0,\{X_2,Y_4,K_6\})$, $(0,\{Y_4,Z_4,K_6\})$, 
$(0,K_6\{X_2,Y_4,Z_4\})$. 
This is free over $A^+(\Gamma_0(2))$ by \cite{aokivector}, so the dimensions are 
given by 
\[
\sum_{k=1}^{\infty}J_{k,1}^-(\Gamma_0(2))=
\frac{t^{19}}{(1-t^2)(1-t^4)^2(1-t^6)}
+\frac{t^{13}+t^{15}+t^{17}}{(1-t^2)(1-t^4)^2(1-t^6)}.
\]
So we proved the case when $\Gamma=\Gamma_0(2)$. 

We sketch the case $N=1$ shortly since the proof is 
simpler than the case $N=2$. 
We have 
\begin{align*}
A^+(\Gamma_2)= & \C[\varphi_4,\varphi_6,\chi_{10},\chi_{12}],\\
A^+(\Gamma_2)=& \chi_{35}A^+(\Gamma_2),
\end{align*}
where $\varphi_k$ are the Eisenstein series of weight $k$ such that 
the constant terms are $1$, and
$\chi_{10}$ and $\chi_{12}$ are unique cusp forms of weight $10$ and $12$, 
such that  the Fourier coeffieients $a(1,1,1;\chi_{10})=
a(1,1,1;\chi_{12})=1$, where 
$a(t_1,t_2,t_{12};f)$ means the Fourier coefficient 
of $f$ at $e(t_1\tau_{11}+t_2\tau_{22}+t_{12}\tau_{12})$. 
In particular, we have 
\[
\chi_{10}=(\theta_{0000}\theta_{0001}\theta_{0010}\theta_{0011}
\theta_{0100}\theta_{0110}\theta_{1000}\theta_{1001}\theta_{1100}\theta_{1111})^2
/4096.
\]
We also define
\[
\chi_{35}=\{\varphi_4,\varphi_6,\chi_{10},\chi_{12}\}/2^9\cdot 3^4.
\]
By the same argument as in the case $N=2$, we have 
\[
\xi(J_{k,1}^+(\Gamma_2))=A_k^+(\Gamma_2)\times A_{k,2}^+(\Gamma_2).
\]
The images $\xi(\Phi)$ of generators are given by
\[
(\varphi_4,0), (\varphi_6,0),(\chi_{10},0),(\chi_{12},0).
\]
 
For $J_{k,1}^-(\Gamma_2)$, the only new thing is to construct 
$\Phi_{35}$ such that $\xi_0(\Phi_{35})=\chi_{35}$, where 
\[
\chi_{35}=\frac{1}{2^9\cdot 3^4}\{\varphi_4,\varphi_6,\chi_{10},\chi_{12}\}. 
\]
Since 
\begin{align*}
W(\p_{12}^2\chi_{10})=& 2\prod_{i=1}^{2}A^6B^6C^6a^6b^6c^6
W((\p_{12}\theta_{1111})^2)/4096,
\\ A^8B^8C^8a^8b^8c^8/32768=&
2\Delta(\tau_{11})\Delta(\tau_{22}),
\end{align*}
where $\Delta(\tau)=e(\tau)\prod_{n=1}^{\infty}(1-e(n\tau))^{24}$ is the 
Ramanujan's $\Delta$ function.
Since we have 
$a(1,1,0;\chi_{12})=10$, $a(1,1,\pm 1;\chi_{12})=1$, we have 
\[
W(\chi_{12})= 
12e(\tau_{11})e(\tau_{22})+\cdots=12\Delta(\tau_{11})\Delta(\tau_{22}).
\]
So we have 
\[
W(\p_{12}^2\chi_{10})=\frac{1}{6}W(\chi_{12}).
\]
So there exists $\Phi_{35}\in J_{35,1}(\Gamma_2)$ 
such that 
\[
\xi(\Phi_{35})=
\left(\chi_{35},
-\frac{(2\pi i)^2}{35\cdot 2^{11}\cdot 3^5} 
\chi_{12}\{\varphi_4,\varphi_6,\chi_{12}\}\right).
\]
Now since $A^-(\Gamma_2)=\chi_{35}A(\Gamma_2)$, 
for any $\Phi\in J_{k,1}(\Gamma_2)$, there exists 
$\Phi_0\in J_{k,1}(\Gamma_2)$ such that 
$\xi(\Phi-\Phi_0)=0$.
So we will determine the images of the shape $(0,h)$.
The forms $h$ is in the space spanned by 
$\{\varphi_4,\varphi_6,\chi_{10}\}$, 
$\{\varphi_4,\chi_{10},\chi_{12}\}$,
$\{\varphi_6,\chi_{10},\chi_{12}\}$ and 
$\{\varphi_4,\varphi_6,\chi_{12}\}$
over $A^+(\Gamma_2)$.  
We have
\[
W(\{\varphi_4,\varphi_6,\chi_{10}\}_{11})=
W(\{\varphi_4,\chi_{10},\chi_{12}\}_{11})=
W(\{\varphi_6,\chi_{10},\chi_{12}\}_{11})=0.
\]
So we have Jacobi forms 
$\Phi_{21}$, $\Phi_{27}$, $\Phi_{29}$ such that 
$\xi(\Phi_{21}) =(0,\{\varphi_4,\phi_6,\chi_{10}\})$, 
$\xi(\Phi_{27})  = (0,\{\varphi_4,\chi_{10},\chi_{12}\})$, 
$\xi(\Phi_{29})  = (0,\{\varphi_6,\chi_{10},\chi_{12}\})$.
So we must ask if we have $W(h_{11})=0$ for
some $h=f\{\varphi_4,\varphi_6,\chi_{12}\}$.
If $(0,h)$ is in the image of $\xi$, then since 
we can check $W(\{\varphi_4,\varphi_6,\chi_{12}\}_{11})
\neq 0$, we should have $W(f)=0$. This is equivalent to say that $f\in \chi_{10}A^+(\Gamma_2)$. 
But we have 
\begin{multline*}
10\chi_{10}\{\varphi_4,\varphi_6,\chi_{12}\}
+4\varphi_4\{\varphi_6,\chi_{12},\chi_{10}\}
\\
+6\varphi_6\{\chi_{12},\chi_{10},\varphi_4\}
+12\chi_{12}\{\chi_{10},\varphi_4,\varphi_6\}
=0.
\end{multline*}
So $J_{*,1}^-(\Gamma_2)$ is 
spanned by $\Phi_{35}$, $\Phi_{21}$, $\Phi_{27}$, $\Phi_{29}$ over $A^+(\Gamma_2)$.
So we prove Theorem 1.1 and Corollary 1.2 for $\Gamma=\Gamma_2$.

(In \cite{ibujacobi}, we claimed there exists a Jacobi form with 
$\xi(\Phi)=(\chi_{35},0)$. But the second component should not be $0$ as we explained above.)

\section{Proof for level $3$}
In cases $N\geq 3$, there are several differences 
from the case $N=1$, $2$. For example, we have 
$I\not\in \Gamma_0(N)^{\psi}$, so for 
$f\in A_k^+(\Gamma_0(N)^\psi)$ and $g\in \Gamma_2$,
the form $f|_k[g]$ is not always an even function of 
$\tau_{12}$ and in general we have 
$W(\p_{12}f^g)\neq 0$, so we need more careful 
consideration.

Now we consider the case $N=3$ in detail.
As in \cite{ibuaoki}, we put
\[
A_2=\begin{pmatrix} 2 & 1 \\ 1 & 2 \end{pmatrix},
\]
\[
E_6=\begin{pmatrix} 2 & -1 & 0 & 0 & 0 & 0 \\
-1 & 2 & -1 & 0 & 0 & 0 \\
0 & -1 & 2 & -1 & 0 & -1 \\
0 & 0 & -1 & 2 & -1 & 0 \\
0 & 0 & 0 & -1 & 2 & 0 \\
0 & 0 & -1 & 0 & 0 & 2 
\end{pmatrix},
\]
\[
E_{6s}=3E_6^{-1}=\begin{pmatrix} 4 & 5 & 6 & 4 & 2 & 3 
\\ 5 & 10 & 12 & 8 & 4 & 6 \\
6 & 12 & 18 & 12 & 6 & 9 \\
4 & 8 & 12 & 10 & 5 & 6 \\
2 & 4 & 6 & 5 & 4 & 3 \\
3 & 6 & 9 & 6 & 3 & 6 \end{pmatrix},
\]
\[
S_4=\begin{pmatrix} 1 & 0 & 3/2 & 0 \\ 
0 & 1 & 0 & 3/2 \\
3/2 & 0 & 3 & 0 \\ 0 & 3/2 & 0 & 3 \end{pmatrix}.
\]
(In \cite{ibuaoki} we write $E_{6s}=E_6^*$ but 
we change the notation since we need superscript for 
other notation.)

For any $\tau\in H_n$ and any $m\times m$ even integral matrix 
$S$, we put 
\[
\theta_{S}(\tau)=\sum_{p\in M_{n,m}(\Z)}e(Tr(pS\,^{t}p \tau)/2).
\]
Then for $\tau\in H_2$, we have 
$\theta_{A_2}(\tau)\in A_2^+(\Gamma_0(3)^{\psi})$
and $\theta_{E_6}(\tau)$, 
$\theta_{E_{6s}}(\tau)\in A_3^{+}(\Gamma_0(3)^{\psi})$. 
We put 
\[
c_4=\sum_{x,y\in \Z^4}(c^2-d^2)e(xS_4\,^{t}x\tau_{11}
+2xS_4\,^{t}y\tau_{12}+yS_4\,^{t}y\tau_{22}),
\]
where $c=(x_1y_3-x_3y_1)+(x_2y_4-y_2x_4)$ and 
$d=(x_1y_4-x_4y_1)+(x_3y_2-x_2y_3)+(x_1y_2-y_1x_2)$.
(This was denoted $\gamma_4$ in \cite{ibuaoki}.)
We have $c_4\in A_4^+(\Gamma_0(3)^{\psi})$.
To make notation consistent with those in the former 
papers \cite{ibulevel3,ibuaoki}, we put
\begin{align*}
\a_1=\alpha_1=\theta_{A_2}(\tau),
\quad \b_3=\theta_{E_6}-10\theta_{A_2}^3+9\theta_{E_{6s}},
\quad \d_3=\theta_{E_6}-9\theta_{E_{6s}}.
\end{align*}
We also put 
\[
\chi_{14}=\frac{1}{2^9\cdot 3^{10}}\{\a_1,\b_3,c_4,\d_3\}
\]
(In \cite{ibuaoki} p. 257 line 2, $\delta_4$ should read 
$\delta_3$).
Then we have 
\begin{align*}
A^+(\Gamma_0(3)^{\psi})=&\C[\a_1,\b_3,c_4,\d_3], \\
A^-(\Gamma_0(3)^{\psi})=& \chi_{14}A^{+}(\Gamma_0(3)^{\psi}).
\end{align*}
For later use, we introduce simpler notation.
We put 
\begin{align*}
a_1 & =\alpha_1, \\
b_3 & = (\beta_3+\delta_3)/2+5\alpha_1^3 =\theta_{E_6},
\\
e_3 & =8\a_1^3+2\b_3-\d_3=\theta_{E_6}
-12\theta_{A_2}^3+27\theta_{E_{6s}},\\
c_4 & =(12\theta_{A_2}\theta_{E_6}-27\theta_{A_2}^4-\theta_{E_8})/162, 
\end{align*}
where $E_8$ is the even unimodular matrix of rank $8$,  
and $\theta_{E_8}=\varphi_4$, where $\varphi_4$ is 
the Siegel Eisenstein series of weight $4$ of level $1$ of degree 2  
with constant term one. Then obviously we have 
$A^+(\Gamma_0(3)^{\psi})=\C[a_1,b_3,e_3,c_4]$. 
We first consider $J_{k,1}^+(\Gamma_0(3)^{\psi})$.
We consider 
\[
Z=\begin{pmatrix} \tau & ^{t}z \\
z & \omega \end{pmatrix} \in H_3
\]
for $\tau\in H_2$, $\omega\in H_1$, 
and $\theta_{A_2}(Z)$ and $\theta_{E_6}(Z)$.
Taking the non-zero Fourier Jacobi coefficients
at $e(\omega)$ of these theta series, we can define 
Jacobi forms $\Phi_1$ of weight $1$ and $\Phi_3$ of 
weight $3$ such that 
$\xi_0(\Phi_1)=a_1$, $\xi_0(\Phi_3)=\theta_{E_6}$.
By \cite{aokivector}, we know that  
$A_{*,2}^+(\Gamma_0(3)^{\psi})$ is generated 
over $A^{+}(\Gamma_0(3)^{\psi})$ by 
$\{\a_1,\b_3\}$, $\{\a_1,\d_3\}$, $\{\a_1,c_4\}$,
$\{\b_3,\d_3\}$, $\{\b_3,c_4\}$, $\{\d_3,c_4\}$.
Since $A_{k,2}(\Gamma_0(3)^{\psi})=0$ for $k\leq 3$,
we have $\xi(\Phi_1)=(a_1,0)$, $\xi(\Phi_3)=(\vartheta_{E_6},0)$.
In the same way, since $A_{4,2}(\Gamma_2)=0$, 
we have $\Phi_4\in J_{4,1}(\Gamma_2)$ such that 
$\xi(\Phi_4)=(\varphi_4,0)$. 
Unfortunately $E_{6s}$ has no vector of length $2$,
so we cannot construct another Jacobi form of 
weight $3$ from a Siegel modular form of degree $3$.
 
We must calculate several images of the Witt operator 
on the divisors of $\chi_{10}=0$.
As written in \cite{ibuaoki}, we have 
\begin{equation}\label{level3ten}
\chi_{10}=c_4e_3^2/6144.
\end{equation}
The $Sp(2,\Z)$ orbit of diagonals is divided into 
two divisors $\c_4=0$ and $e_3=0$.  
We write $M_1$ as in Lemma \ref{level2rep}
and $J=\begin{pmatrix} 0 & -1_2 \\ 1_2 & 0 \end{pmatrix}$.
We put  
\[
K=-I(M_1J)I=\begin{pmatrix} 0 & -1_2 \\ 1_2 & S_0 \end{pmatrix}, \quad S_0=\begin{pmatrix} 0 & 1 \\ 1 & 0 \end{pmatrix}.
\]
\begin{lemma}
We have 
\[
\Gamma_0(3)\backslash Sp(2,\Z)/\Gamma_0=
\{1_4,M_1\}=\{1_4,K\}.
\]
\end{lemma} 

\begin{proof}
In the representatives $M$ in Lemma \ref{rightrep},
those which are not in $\Gamma_0$ are given by 
$b=\pm 1$. By seeing $IMI$ where $I={\rm diag}(1,-1,1,-1)$, 
we may assume that $b=1$. If we put 
$M_0=\begin{pmatrix} 0 & 1_2 \\ -1_2 & S \end{pmatrix}$
for $S=\begin{pmatrix} 0 & 1 \\ 1 & 0 \end{pmatrix}$,
then we have 
\begin{align*}
M_0^{-1}\begin{pmatrix} 0 & 0 & 1 & 0 \\
0 & 0 & 0 & 1 \\
-1 & 0 & a & 1 \\
0 & -1 & 1 & c \end{pmatrix}
& =
\begin{pmatrix}1 & 0 & -a & 0 \\ 0 & 1 & 0 & c \\
0 & 0 & 1 & 0 \\ 0 & 0 & 0 & 1 \end{pmatrix}\in \Gamma_0,
\\
M_0^{-1}\begin{pmatrix} 0 & 0 & 1 & 0 \\
0 & 1 & 0 & 0 \\ -1 & 1 & a & 0 \\
0 & 0 & 1 & 1 \end{pmatrix}
& =\begin{pmatrix} 1 & 0 & -a & 0 \\
0 & 0 & 0 & -1 \\ 0 & 0 & 1 & 0 \\ 0 & 1 & 0 & 0 \end{pmatrix}
\in \Gamma_0.
\end{align*}
So the representatives are $1_4$ and $M_0$, but 
$M_1=M_0J_2$ and $J_2\in \Gamma_0$. 
\end{proof}

For any $f \in A_k^+(\Gamma_0(3)^{\psi})$,  we write 
\[
f^*=f|_k[K].
\]
\subsection{Structure of $J_{*,1}^+(\Gamma_0(3)^{\psi})$}
Since $\xi(\Phi_1)=(a_1,0)$ and $\xi(\Phi_3)=(b_3,0)$ and $\xi_4(\Phi_4)=(\varphi_4,0)$, 
by \eqref{wittcondition}, 
we have $W(\p_{12}a_1)=W(\p_{12}a_1^*)=W(\p_{12}b_3)=W(\p_{12}b_3^*)=0$
$W(\p_{12}\varphi_4)=0$. ($\varphi_4^*=\varphi_4$. In this case, this is an 
even function of $\tau_{12}$ so the above fact is also clear from this.)
Since $\theta_{S}(\tau)$ is obviously an even function of $\tau_{12}$ for any 
even positive definite symmetric matrix $S$, we have
also $W(\p_{12}e_3)=0$.
By direct calculation, we can show that $W(c_4)=0$. 
By \eqref{level3ten}, since $\chi_{10}^*=\chi_{10}$ and $W(\chi_{10})=0$,
we have 
\[
W(c_4^*)W(e_3^*)^2=0.
\]
Since we can show directly that $W(e_3)\neq 0$, so we have $W(\p_{12}c_4)=0$.
But if $W(c_4^*)=0$, then $W(c_4|[M])=0$ for all $M\in Sp(2,\Z)$,
so $c_4$ is divisible by $\chi_5$, which is a contradiction since $c_4/\chi_5$ 
is of weight $-1$. So we have 
$W(c_4^*)\neq 0$ and $W(e_3^*)=0$. So we have $W((e_3^2)^*)=0$ and 
$W(\p_{12}((e_3^2)^*)=2W(\p_{12}e_3^*)W(e_3^*)=0$, so 
there exists $\Phi_6\in J_{6,1}^{I}(\Gamma_0(3)^{\psi})$ such that 
$\xi(\Phi_6)=(e_3^2,0)$. 
By the way,  we have 
\[
2\Delta(\tau_{11})\Delta(\tau_{22})=W(\p_{12}^2\chi_{10})=2W(c_4^*)W((\p_{12}e_3^*)^2),
\]
where $\Delta$ is the Ramanujan delta.
So we have $W(\p_{12}e_3^*)\neq 0$, which will be used later. 

Now for $\Gamma=\Gamma_0(3)^{\psi}$, we consider the module
\[
\fM=A^+(\Gamma)a_1+A^+(\Gamma)b_3+
A^+(\Gamma)e_3^2+A^+(\Gamma)\varphi_4.
\]
This expression is not free, but 
we will show that $\xi_0(J_{*,1}(\Gamma_0(3)^{\psi})
=\fM$. 
Since $W(\p_{12}e_3)\neq 0$ and $A_{3,2}(\Gamma_0(3)^{\psi})=0$, it is clear that there is no 
$\Phi\in J_{3,1}(\Gamma_0(3)^{\psi})$ such that 
$\xi(\Phi)=(e_3,*)$. It is clear that the submodule 
of $\C[a_1,b_3,e_3,\varphi_4]$ consisting of those polynomials with no
constant term and no term of the form $ce_3$ where $c$ is a constant
is given by $\fM$. 
By the relation \eqref{moduleaction}, 
we have $\xi(f\Phi_1)=(fa_1,*)$, $\xi(f\Phi_3)=(fb_3,*)$.
$\xi(f\Phi_4) =(f\varphi_4,*)$, $\xi(f\Phi_6)=(fe_3^2,*)$, 
so it is clear that 
\[
\xi_0(J_{*,1}^{+}(\Gamma_0(3)^{\psi}))=\xi(\fM).
\]
So for $\phi\in J_{*,1}^{+}(\Gamma_0(3)^{\psi})$, 
there exists 
\[
\Phi_0\in A^+(\Gamma_0(3)^{\psi})\Phi_1+
A^+(\Gamma_0(3)^{\psi})\Phi_3+
A^+(\Gamma_0(3)^{\psi})\Phi_4+
A^+(\Gamma_0(3)^{\psi})\Phi_6
\]
such that 
$\xi_0(\Phi-\Phi_0)=0$, so we may assume that
$\xi_0(\Phi)=0$. 
We can show that by \eqref{rel1}, 
there exists $\Phi\in J_{*,1}(\Gamma_0(3)^{\psi})$ such that 
$\xi(\Phi)=(0,h)$ for $h=\{a_1,b_3\}$, $\{a_1,\varphi_4\}$, 
$\{b_3,\varphi_4\}$, $\{a_1,e_3^2\}$, $\{b_3,e_3^2\}$,
$\{\varphi_4,e_3^2\}$. 
We must show that $(0,h)$ should be generated by 
these forms.
To see this we must determine $h\in A_{k,2}(\Gamma_0(3)^{\psi})$ such that $W(h_{11})=W(h^*_{11})=0$.
We have already explained that 
\begin{align}\label{wittzero}
& W(\p_{12}a_1)=W(\p_{12}b_3)=W(\p_{12}a_1^*)=
W(\p_{12}b_3^*)=
W(\p_{12}e_3)
\\ & =W(\p_{12}\varphi_4)=
W(\p_{12}\varphi_4^*)=0, \notag
\end{align}
but $W(\p_{12}e_3^*)\neq 0$. 
By \eqref{wittzero}, for $h=\{a_1,b_3\}$, $\{a_1,\varphi_4\}$ or $\{b_4,\varphi_4\}$,
we have $W(h_{11})=W(h_{11}^*)=0$.
Now put 
\begin{equation}\label{defh}
h=F_1\{a_1,e_3\}+F_2\{b_3,e_3\}+F_3\{\varphi_4,e_3\}.
\end{equation}
By \eqref{wittzero}, we have 
\[
W(\{a_1,e_3\}_{11})=W(\{b_3,e_3\}_{11})=
W(\{\varphi_4,e_3\}_{11})=0,
\]
so we have $W(h_{11})=0$. 
On the other hand, if we assume that $W(h_{11}^*)=0$, 
then we have
\[
W(F_1^*\{a_1^*,e_3^*\}_{11}
+F_2^*\{b_3^*,e_3^*\}_{11}
+F_3^*\{\varphi_4,e_3^*\}_{11})=0.
\]
By \eqref{wittzero}, this means 
\[
W(F_1^*a_1^*+3F_2^*b_3^*+4F_3^*\varphi_4)W(\p_{12}e_3^*)=0.
\]
Since $W(\p_{12}e_3^*)\neq 0$,  we have 
\begin{equation}\label{wittrelation}
W(F_1^*a_1^*+3F_2^*b_3^*+4F_3^*\varphi_4)=0.
\end{equation}
Since the kernel of $W$ on 
$\C[a_1^*,b_3^*,\varphi_4,e_3^*]$ is the ideal 
generated by $e_3^*$, the relation \eqref{wittrelation}
means that 
\[
F_1^*a_1^*+3F_2^*b_3^*+4F_3^*\varphi_4=G^*e_3^*
\]
for some $G\in A^+(\Gamma_0(3)^{\psi})$.
Since $f\rightarrow f^*$ is an isomorphism, 
the above relation is equivalent to the relation
\begin{equation}\label{level3relation}
F_1a_1+3F_2b_3+4F_3\varphi_4=Ge_3.
\end{equation}
Now we will see what are $F_i$.
Since $a_1$, $b_3$, $\varphi_4$, $e_3$ 
are algebraically independent, $A^+(\Gamma_0(3)^{\psi})
=\C[a_1,b_3,\varphi_4,e_3]$ is a weighted 
polynomial ring, for any $i=1$, $2$, $3$, we may uniquely write 
\begin{equation}
F_i=e_3G_i(a_1,b_3,e_3,\varphi_4)+H_i(a_1,b_3,\varphi_4).
\end{equation}
Then by \eqref{level3relation}, we have 
\begin{equation}\label{level3relH}
a_1H_1+3b_3H_2+4\varphi_4H_4=0.
\end{equation}
For $i=1$, $2$, we also write 
\[
H_i(a_1,b_3,\varphi_4)=P_i(a_1,b_3)+\varphi_4Q_i(a_1,b_3,\varphi_4).
\]
for some polynomials $P_i$, $Q_i$.
So \eqref{level3relation} means 
\begin{align*}
& a_1P_1(a_1,b_3)+3b_3P_2(a_1,b_3)  = 0, \\
& a_1Q_1(a_1,b_3,\varphi_4)+3b_3Q_2(a_1,b_3,\varphi_4)+4H_3(a_1,b_3,\varphi_4)=0.
\end{align*}
This means that 
$P_1(a_1,b_3)=-3b_3R(a_1,b_3)$ and $P_2(a_1,b_3)=a_1R(a_1,b_3)$ 
for some polynomial $R$. 
So rewriting $h$ defined by \eqref{defh} by using these 
expressions, we have 
\begin{align*}
h = &  G_1e_3\{e_3,a_1\}+G_2e_3\{e_3,b_3\}+G_3e_3\{e_3,\varphi_4\}
\\ & +R(-3b_3\{e_3,a_1\}+a_1\{e_3,b_3\})
\\ & +\frac{Q_1}{4}(4\varphi_4\{e_3,a_1\}-a_1\{e_3,\varphi_4\})
+\frac{Q_2}{4}(4\varphi_4\{e_3,b_3\}-3b_3\{e_3,\varphi_4\}.
\end{align*}
Here by the Jacobi identity \eqref{rel2}, we have 
\begin{align*}
 -3b_3\{e_3,a_1\}+a_1\{e_3,b_3\} &=3e_3\{a_1,b_3\}, 
\\  4\varphi_4\{e_3,a_1\}-a_1\{e_3,\varphi_4\} &=
3e_3\{\varphi_4,a_1\},
\\ 4\varphi_4\{e_3,b_3\}-3b_3\{e_3,\varphi_4\} &=
3e_3\{\varphi_4,b_3\}.
\end{align*}
So we see that $h$ is spanned over $A^*(\Gamma_0(3)^{\psi})$ by 
\begin{equation}\label{generator3}
\{a_1,b_3\}, \{a_1,\varphi_4\}, \{b_3,\varphi_4\},
e_3\{a_1,e_3\}, e_3\{b_3,e_3\}, e_3\{\varphi_4,e_3\}.
\end{equation}
The latter three forms are equal to 
$\{a_1,e_3^2\}/2$, $\{b_3,e_3^2\}/2$, $\{\varphi_4,e_3^2\}$.
Now by \eqref{rel1}, if there are Jacobi forms $\Phi_f$ and $\Phi_g$ such that $\xi(\Phi_f)=(f,0)$ and $\xi(\Phi_g)=(g,0)$, then $\xi(f\Phi_g-g\Phi_f)=
const (0,\{f,g\})$. 
So for any form $h$ in \eqref{generator3}, and also 
for any form $h$ in $A_{*,2}(\Gamma_0(3)^{\psi})$ 
spanned by forms in \eqref{generator3} over $A^+(\Gamma_0(3)^{\psi})$, 
there exists an element $\Phi$ in the module 
spanned by 
$\Phi_1$, $\Phi_3$, $\Phi_4$, $\Phi_6$ over 
$A(\Gamma_0(3)^{\psi})$ such that $\xi(\Phi)=(0,h)$. 
This means that we have 
\[
J_{*,1}^*(\Gamma_0(3)^{\psi})
=
A^+(\Gamma_0(3)^{\psi})\Phi_1+
A^+(\Gamma_0(3)^{\psi})\Phi_3+
A^+(\Gamma_0(3)^{\psi})\Phi_4+
A^+(\Gamma_0(3)^{\psi})\Phi_6.
\]
Now to show that $\Phi_1$, $\Phi_3$, $\Phi_4$, $\Phi_6$ 
are free generators of $J_{*,1}^{+}(\Gamma_0(3)^{\psi})$ 
over $A^+(\Gamma_0(3)^{\psi})$, we use the following Lemma.
 
\begin{lemma}\label{freebasis}
We write $\Gamma=\Gamma_2$, $\Gamma_0(2)$, 
$\Gamma_0(3)^{\psi}$, $\Gamma_0(4)^{\psi}$, or $\Gamma_0^0(2)^{\psi}$,
Assume that $G_i\in A_{k_i}^{I}(\Gamma)$ for 
$i=1$, $2$, $3$, $4$. If 
$\{G_1,G_i\}$ ($i=2$, $3$, $4$) are free over $A^{I}(\Gamma)$, 
then $(G_i,0)$ for $1\leq i\leq 4$ are free basis over $A^I(\Gamma)$.
\end{lemma}

\begin{proof}
Assume that 
\[
\sum_{i=1}^{4}F_i(G_i,0)=0
\]
for $F_i\in A^I(\Gamma)$. 
It is enough to show that $F_i=0$. We have 
By \eqref{moduleaction}, we have 
\[
F(G,0)=\biggl(FG,\frac{(2\pi i)^2}{k(k+l)}\{F,G\}\biggr)
\]
for $F\in A_{k}(\Gamma)$ and $G \in A_{l}^I(\Gamma)$.
Since weight $F_i$ plus weight $G_i$ are the same, we have 
\begin{align}
\sum_{i=1}^{4}F_iG_i & =0, \label{shiki1} \\
\sum_{i=1}^{4}\{F_i,G_i\}/k_i & = 0. \label{shiki2}
\end{align}
We have 
\[
\{F_iG_i,G_1\}=G_i\{F_i,G_1\}+F_i\{G_i,G_1\}.
\]
By \eqref{shiki1} and by $\{G_1,G_1\}=0$, we have 
\[
G_1\{F_1,G_1\}+\sum_{j=2}^{4}(G_i\{F_i,G_1\}+F_i\{G_i,G_1\})=0.
\]
Subtracting \eqref{shiki2} times $k_1G_1$ from this, we have 
\[
\sum_{j=2}^{4}(F_j\{G_j,G_1\}+G_j\{F_j,G_1\}-\frac{k_1}{k_j}G_1\{F_j,G_j\})=0.
\]
By the relation \eqref{rel2},  we have 
\[
k_1G_1\{F_j,G_j\}-k_jG_j\{F_j,G_1\}+l_jF_j\{G_j,G_1\}=0, 
\]
where $l_j$ is the weight of $F_j$. So erasing $\{F_j,G_1\}$ and 
$\{F_j,G_j\}$, we have 
\[
\sum_{j=2}^{4}(1+\frac{l_j}{k_j})F_j\{G_j,G_1\}=0.
\]
But $\{G_j,G_1\}$ are linearly independent over $A^I(\Gamma)$ by 
the assumption, so $F_j=0$. 
\end{proof} 

Now if we put  
$G_2=a_1$, $G_3=b_3$, $G_4=\varphi_4$, $G_1=e_3^2$ 
in Lemma \ref{freebasis}, then $\{a_1,e_3^2\}$, $\{b_3,e_3^2\}$, 
$\{\varphi_4,e_3^2\}$ 
are linearly independent over $A^{+}(\Gamma_0(3)^{\psi})$.
(See \cite{aokivector}, noting that 
$\{f,e_3^2\}=2e_3\{f,e_3\}$.) Or we can show this directly 
by regarding any $\{f,g\}$ as a vector of coefficients of 
$u_1^2$, $u_1u_2$, $u_2^2$ and seeing that the Fourier 
expansion of 
the determinant of the $3\times 3$ matrix
\[
(\{a_1,e_3\}, \{b_3,e_3\}, \{\varphi_4,e_3\})
\]
is given by 
\[
 16874416668672(e(-\tau_{12})-e(\tau_{12}))(e(2\tau_{11}+3\tau_{22})+ e(2\tau_{22}+3\tau_{11}))+\cdots \neq 0.
 \]

\begin{proposition}
As a $A^+(\Gamma_0(3)^{\psi})$ module, 
$J_{*,1}(\Gamma_0(3)^{\psi})$ is spanned by 
$\Phi_1$, $\Phi_3$, $\Phi_4$, $\Phi_6$. 
These are free generators and dimensions are given by 
\[
\sum_{k=1}^{\infty}
\dim J_{k,1}^{+}(\Gamma_0(3)^{\psi})t^{k}
=
\frac{t+t^3+t^4+t^6}
{(1-t)(1-t^3)^2(1-t^4)}.
\]
\end{proposition}

\subsection{Structure of $J_{*,1}^-(\Gamma_0(3)^{\psi})$}
To avoid inessential constants
each time, we put 
\begin{equation}
X_{14}=\{a_1,b_3,c_4,e_3\}=2^8\cdot 3^{11}\chi_{14}.
\end{equation}
First we show that there exists a Jacobi form 
$\Phi_{14}\in J_{14,1}(\Gamma_0(3)^{\psi})$ such that 
$\xi_0(\Phi_{14})=X_{14}$. 
To show this, we first consider condition the diagonal orbit for $M=1_4$ and calculate $W(p_{12}X_{14})$. 
By \cite{ibuaoki} p. 258, we have 
\[
W(c_4)=W(\p_{ij}c_4)
=W(\p_{12}\p_{ii}c_4)=0 \qquad (i,j=1,2).
\]
So $W(\{\p_{12}a_1,b_3,c_4,e_3\})=0$
since the components of the first column and the third column are $0$ except for the third row. Similarly we have $W(\{a_1,\p_{12}b_2,c_4,e_3\})
=W(\{a_1,b_3,c_4,\p_{12}e_3\})=0$. The third row of $W(\{a_1,b_3,\p_{12}c_4,e_3\})$ is zero except for the (3,3) component, so we have 
\[
W(\p_{12}X_{14})=W(\p_{12}^2c_4)\{a_1,b_3,e_3\}_{11}/2.
\]
Here $W(\p_{12}^2c_4)$  is in the symmetric tensor of
$A_6(\Gamma_0^{(1)}(3))$, since $2\p_{12}^2-\p_{11}\p_{22}$ 
is a differential operator preserving automorphy under the Witt 
operator increasing the weight by $2$, and $W(\p_{11}\p_{22}c_4)=0$. 
For $\tau_{ii}\in H_1$, we put 
\begin{align*}
\F_1(\tau_{11})& =\theta_{A_2}(\tau_{11})=
\sum_{x\in \Z^2}e(xA_2\,^{t}x\tau_{11}/2), 
\\ 
\F_2(\tau_{11})& =\theta_{E_6}(\tau_{11})=
\sum_{x\in \Z^6}e(xE_6\,^{t}x\tau_{11}/2).
\end{align*}
Then $\oplus_{k=0}^{\infty}A_k(\Gamma_0^{(1)}(3)^{\psi})
=\C[\F_1,\F_2]$. (See \cite{ibuaoki} p. 258. Although $\F_2$ is of weight $3$, we do not write it as $\F_3$ to make it consistent with 
notation in \cite{ibuaoki}.)
We also put $\G_1=\F_1(\tau_{22})$, $\G_2=\F_2(\tau_{22})$.
By \cite{ibuaoki} loc., cit., we have 
\begin{align*}
W(a_1) & = \F_1\G_1, \\
W(\theta_{E_6}) & = \F_2\G_2, \\
W(\theta_{E_{6s}})& = (4\F_1^3-\F_2)(3\G_1^3-\G_2)/9, 
\\
W(e_3) & = 4(3\F_1^3-\F_2)(3\G_1^3-\G_2).
\end{align*}

Calculating the Fourier coefficients, or seeing that 
$W(\p_{12}^2\chi_{10})=2\Delta(\tau_{11})\Delta(\tau_{22})$ and 
$\Delta(\tau_{11})=-(\F_1^3-\F_2)(3\F_1^3-\F_2)/432$,  we see that 
\begin{align*}
W(\p_{12}^2c_4)& =
\frac{1}{3^5}(\F_1^3-\F_2)(3\F_1^3-\F_2)
(\G_1^3-\G_2)(3\G_1^3-\G_2)
\\ & =W(e_3)\times \frac{1}{2^2\cdot 3^5}(\F_1^3-\F_2)(\G_1^3-\G_2).
\end{align*}
If we put 
\[
f_3=-12\a_1^3+3\theta_{E_6}+9\theta_{E_{6s}}, 
\]
then we have 
\[
W(f_3)=4(\F_1^3-\F_2)(\G_1^3-\G_2),\quad 
W(\p_{12}^2c_4)=2^{-4}\cdot 3^{-5}W(e_3f_3).
\]
This means that we have 
\[
W\left(\frac{2\pi i}{14}\frac{\p X_{14}}{\p \tau_{12}}
-\frac{(2\pi i)^2}{14\cdot 2^5\cdot 3^5}e_3f_3\{a_1,b_3,e_3\}_{11}\right)
=0,
\]
so the condition \eqref{wittcondition} is satisfied for 
$(X_{14},\frac{(2\pi i)^2}{14}h_0)$ for $M=1_4$, where 
$h_0=-2^{-5}\cdot 3^{-5}e_3f_3\{a_1,b_2,e_3\}$ for $M=1_4$.

Next we calculate $W(\p_{12}X_{14}^*)$. Here 
we write $f^*=f|_{k}[N]$ as before.
To see the Witt image, we change the expression of $X_{14}$
slightly. We have $\varphi_4^*=\varphi_4$ and 
$W(\p_{ij}\varphi_4^*)=0$ for any $i$, $j$ ($1\leq i,j\leq 2$).
In order to use this fact, we change $c_4$ by $\varphi_4$ as follows.
By \cite{ibuaoki} p. 259, we have 
\[
c_4=\frac{1}{162}(-27a_1^4+12a_1b_3+a_1e_3-\varphi_4).
\]
So we have
\[
X_{14}=-\frac{1}{162}\{a_1,b_3,\varphi_4,e_3\}.
\]
We have 
\begin{align*}
W(\p_{12}a_1^*) & =W(\p_{12}\p_{11}a_1^*)=W(\p_{12}\p_{22}a_1^*)=0, \\
W(\p_{12}b_3^*)& = W(\p_{12}\p_{11}b_3^*)=W(\p_{12}\p_{22}b_3^*)=0, \\
W(\p_{12}\varphi_4) & = W(\p_{12}\p_{11}\varphi_4)=W(\p_{12}\p_{22}\varphi_4)=0,\\
W(e_3^*) & = W(\p_{11}e_3^*)=W(\p_{22}e_3^*)=0.
\end{align*}
So we have
\[
W(\{\p_{12}a_1^*,b_3^*,\varphi_4,e_3^*\})=
W(\{a_1^*\p_{12}b_3^*,\varphi_4,e_3^*\})=
W(\{a_1^*,b_3^*,\p_{12}\varphi_4,e_3^*\})=0,
\]
since the $(i,4)$ components for $i\neq 3$ are all 
$0$ and the components of the derivatives by $\p_{12}$ of 
the $j$-th column of $\{a_1^*,b_3^*,c_4^*,\varphi_4\}$ 
for $j\neq 4$ are all $0$ except for the third component.  
The third row of $\{a_1^*,b_3^*,\varphi_4,\p_{12}e_3^*\}$ 
are all $0$ except for the fourth column, so we have  
\begin{align*}
W\biggl(
\p_{12}\{a_1^*,b_3^*,\varphi_4,e_3^*\}\biggr)
 & =-W(\p_{12}^2e_3^*)W(\{a_1^*,b_3^*,\varphi_4\}_{11})/2, 
\\  
W(\p_{12}X_{14}) & 
=\frac{1}{2\cdot 162}W(\p_{12}e_3^*)W(\{a_1^*,b_3^*,\varphi_4\}_{11}).
\end{align*}
Here $W(\p_{12}^2e_3^*)$ is a tensor of modular 
forms of one variable belonging to some
discrete group by the same reason as before.

We first explain how to obtain $e_3^*$.
We have 
\[
K=JM_2,\quad M_2=\begin{pmatrix} 1_2 & S_0 \\ 0 & 1_2 
\end{pmatrix} \text{ for } S_0=\begin{pmatrix} 0 & 1 \\ 1 
& 0 \end{pmatrix}, \quad J=\begin{pmatrix} 0 & -1_2 \\ 1_2 & 0 \end{pmatrix}.
\]
Here the Fourier expansion of $f|[M_2]$ 
is easily obtained by the Fourier expansion of 
$f$, so the problem is to write down $f|[J]$.
By \cite{andrianovmaloletkin}, for an even positive integer $m$ and for any $m\times m$
symmetric matrix $S$ and $\tau\in H_n$, 
we have 
\begin{equation}
\theta_{S}(-\tau^{-1})=\det(-i\tau)^{m/2}\det(S)^{-n/2}
\theta_{S^{-1}}(\tau),
\end{equation}
so we have 
\begin{equation}\label{inverseformula}
\theta_{S}|_{m/2}[J]=\det(\tau)^{-m/2}\theta_{S}(-\tau^{-1})
=(-i)^{nm}\det(S)^{-n/2}\theta_{S^{-1}}(\tau).
\end{equation}
Applying this formula for $n=2$, we have 
\begin{align*}
(\theta_{A_2}|_1[J])(\tau)& =-(1/3)\vartheta_{A_2/3}
(\tau)=(-1/3)\theta_{A_2}(\tau/3), 
\\
(\theta_{E_6}|_3[J])(\tau)& = -(1/3)\theta_{E_6^{-1}}(\tau)
=-(1/3)\theta_{E_{6s}/3}(\tau)=(-1/3)\theta_{E_{6s}}(\tau/3), \\
(\theta_{E_{6s}}|_{3}[J])(\tau)& = -(1/3^5)\theta_{E_{6s}^{-1}}(\tau)=-(1/3^5)\theta_{E_6}(\tau/3),\\
\varphi_4|_4[J]& =\varphi_4.
\end{align*}
So it is not difficult to calculate 
many Fourier coefficients of $W(f^*)$ 
for concrete $f$.

We determine the discete group to which $W(f^*)$ should belong.
For $P_i=\begin{pmatrix} p_i & q_i \\ r_i & s_i \end{pmatrix}
\in SL_2(\Z)$, we should consider the condition
\[
K\begin{pmatrix} p_1 & 0 & q_1 & 0 \\ 0 & p_2 & 0 & q_2 \\ r_1 & 0 & s_1 & 0 \\ 0 & r_2 & 0 & s_2 
\end{pmatrix}
K^{-1}
=
\begin{pmatrix} s_1 & -r_1 & -r_1 & 0 \\
-r_2 & s_2 & 0  & -r_2 \\
r_2-q_1 & p_1-s_2 & p_1 & r_2 \\ 
p_2-s_1 & r_1-q_2 & r_1 & p_2 
\end{pmatrix}
\in \Gamma_0(3).
\]
This means 
\[
P_2\equiv S_0P_1S_0 \bmod 3 \qquad \text{ for } S_0=\begin{pmatrix}
0 & 1 \\1 &0 \end{pmatrix}.
\]
So we put 
\[
\Gamma_{S_0}(3)=\{(P_1,P_2); P_i\in SL_2(\Z), P_2
\equiv S_0P_1S_0\bmod 3\}.
\]
For $(P_1,P_2)\in \Gamma_{S_0}(3)$, we have 
\[
\begin{vmatrix} p_1 & r_2 \\ r_1 & p_2 \end{vmatrix}
\equiv p_1s_1-r_1q_1 = 1\bmod 3.
\]
So we have 
$K\iota(P_1,P_2)K^{-1}\in \Gamma_0(3)^{\psi}$ and 
$f^*|_k[\iota(P_1,P_2)]=f^*$. 
Now we put $U=\begin{pmatrix} S_0 & 0 \\ 0 & S_0 \end{pmatrix}$.
Then we have $KU=UK$. 
Since $UI\in \Gamma_0(3)^{\psi}$, we have 
$f|[U]=f|[I]$, and we also have $f|[I]=(-1)^kf$ because 
$f\in A_k^+(\Gamma_0(3)^{\psi})$.  
So we have 
\[
f^*|[U]=f|[K][U]=f|[U][K]=f|[I][K]=(-1)^k f|[K]=(-1)^k f^*.
\]
This means that 
\[
f^*\begin{pmatrix} \tau_{22} & \tau_{12} \\ \tau_{12} & \tau_{11} \end{pmatrix}=
f^*\begin{pmatrix} \tau_{11} & \tau_{12} \\ \tau_{12} & \tau_{22} \end{pmatrix}.
\]
By the above considerations, we have the following two conditions. \\
(1) $W(f^*)|_{k}[\iota(A_1,A_2)]=
W(f^*)$ for $(A_1,A_2)\in \Gamma_3$. \\
(2) We have 
$W(f^*)(\tau_{22},\tau_{11})=W(f^*)(\tau_{11},\tau_{22})$.

If we put 
$\Gamma(3)=\{g\in SL_2(\Z);g\equiv 1_2 \bmod 3\}$,
then we have $\Gamma(3)\subset \Gamma_{S_0}(3)$. 
Using the notation in \cite{ebeling} p.135, we put 
\begin{align*}
\theta_0(\omega) &= \sum_{x,y\in \Z}q^{x^2-xy+y^2}
=1+6(q+q^3+q^4+\cdots), \\
\theta_{1}(\omega)&=q^{1/3}\sum_{z,y\in \Z}
q^{x^2-xy+y^2+x-y}=3q^{1/3}(1+q+2q^2+\cdots),
\end{align*}
where we write $q=e(\omega)$ for $\omega\in H_1$. 
Here we have $\theta_0(\omega)=\theta_{A_2}(\omega)$. 
Then by \cite{ebeling} Theorem 5.4,  we have 
\[
A(\Gamma(3))=\C[\theta_0,\theta_1].
\]
For the sake of simplicity, we put 
\[
x_0=\theta_0(\tau_{11}),\quad x_1=\theta_1(\tau_{11}),
\quad y_0=\theta_0(\tau_{22}), 
\quad y_1=\theta_{1}(\tau_{22}).
\]
By calculating sufficiently many Fourier coefficients, we can show that 
\begin{align*}
W(a_1^*)= & -(x_0y_0+2x_1y_0+2x_0y_1-2x_1y_1)/3, \\
W(b_3^*) = & -(x_0^3 y_0^3 + 6 x_0 x_1^2 y_0^3 
+ 2 x_1^3 y_0^3 + 6 x_0^3 y_0 y_1^2 - 
 18 x_0 x_1^2 y_0 y_1^2 + 12 x_1^3 y_0 y_1^2 
 \\ & + 2 x_0^3 y_1^3 + 12 x_0 x_1^2 y_1^3 + 4 x_1^3 y_1^3)/3,\\
W(e_3^*)  =  & \,0, \\
W(\varphi_4)  = & (x_0^4+8x_0x_1^3)(y_0^4+8y_0y_1^3), \\
W(c_4^*)= & \frac{8}{81}x_1y_1(x_0-x_1)(y_0-y_1)
(x_0^2+x_0x_1+x_1^2)(y_0^2+y_0y_1+y_1^2).
\end{align*}
On the other hand, $W(\p_{12}^2e_3^*)$ 
belongs to the symmetric tensor of $A_5(\Gamma(3))$.
Indeed as explained in \eqref{EZdiff},
the function $W(((3/2)\p_{12}^2-\p_{11}\p_{22})e_3^*)$ is modular, but 
$W(\p_{ii}e_3^*)=W(\p_{11}\p_{22}e_3^*)=0$, so $W(\p_{12}^2e_3^*)$ is modular. 
More concretely, by using the relation  
\[
e_3^*=\theta_{E_6}^*-12(\theta_{A_2}
^{*})^3+27\theta_{E_{6s}}^*,
\]
we have    
\begin{align*}
W(\p_{12}^2e_3^*)
= & 
\frac{16}{9}(x_0 y_0 +   2 x_1 y_0 + 2 x_0 y_1 - 2 x_1 y_1) \times
\\ & x_1y_1(x_0 -x_1)(y_0-y_1) (x_0^2 + x_0 x_1 + x_1^2)  
(y_0^2+y_0y_1+y_1^2)
\\ = &  -2\cdot 3^3 W(a_1^*c_4^*).
\end{align*}
So if we define 
\[
H=\frac{(2\pi i)^2}{14}\biggl(-\frac{1}{2^5\cdot 3^5}e_3f_3\{a_1,b_3,e_3\}+\frac{(2\pi i)^2}{14}\cdot \frac{3^3}{162} a_1c_4\{a_1,b_3,\varphi_4\}\biggr)\in A_{14,2}(\Gamma_0(3)^{\psi}),
\]
then since $W(e_3^*)=W(c_4)=0$, a pair $(X_{14},H)$
satisfies the condition \eqref{wittcondition} both for $M=1_4$ and $K$, so there exists $\Phi_{14}$ such that 
$\xi(\Phi_{14})=(X_{14},H)$.
Now for any form $\Phi \in J_{k,1}^{-}(\Gamma_0(3)^{\psi})$,
we have $\xi_0(\Phi)\in A^{-}(\Gamma_0(3)^{\psi})=X_{14}A^+(\Gamma_0(3)^{\psi})$, 
so we may assume that $\xi_0(\Phi-f\Phi_{14})=0$ 
for some $f\in A^+(\Gamma_0(3)^{\psi})$.
So we will determine $\Phi$ such that $\xi_0(\Phi)=0$.
This is equivalent to determine $h\in A_{*,2}^{-}(\Gamma_0(3)^{\psi})$
such that  $W(h_{11})=W(h^*_{11})=0$. 
By \cite{aokivector}, there exists $F_i \in A^+(\Gamma_0(3)^{\psi})$ such that 
\[
h=F_1\{a_1,b_3,e_3\}+F_2\{a_1,b_3,c_4\}+F_3\{a_1,c_4,e_3\}+F_4\{b_3,c_4,e_3\}.
\]
Since $W(c_4)=W(\p_{ii}c_4)=W(e_3^*)=W(\p_{ii}e_3^*)=0$ for $i=1$, $2$, it is clear by definition that 
\[
W(\{f,c_4,e_3\}_{11})=W(\{f^*,c_4^*,e_3^*\}_{11})=0
\]
for any $f \in A^+(\Gamma_0(3)^{\psi})$. 
In particular, there exist $\Phi_{9}\in J_{9,1}^-(\Gamma_0(3)^{\psi})$ and $\Phi_{11}
\in J_{11,1}(\Gamma_0(3)^{\psi})$ such that 
\[
\xi(\Phi_9)=(0,\{a_1,c_4,e_3\}),  
\quad 
\xi(\Phi_{11})=(0,\{b_3,c_4,e_3\}).
\]
So we may assume $F_3=F_4=0$ 
and the condition becomes
\begin{align*}
W(h_{11})& =W(F_1)W(\{a_1,b_3,e_3\}_{11})=0, \\
W(h_{11}^*)&=W(F_2^*)W(\{a_1^*,b_3^*,c_4^*\}_{11})=0.
\end{align*}
By seeing the Fourier expansion, we can show that 
\begin{align*}
& W(\{a_1,b_3,e_3\}_{11}/2)=1889568(e(\tau_{22}+2\tau_{11})-e(\tau_{11}+2\tau_{22}))+\cdots \neq 0,  
\\ & W(\{a_1^*,b_3^*,c_4^*\}_{11}/2)
=\frac{16}{243}(-e((2\tau_{11}+\tau_{22})/3)+e((2\tau_{22}+\tau_{11})/3))+\cdots \neq 0. 
\end{align*}
This means that $W(F_1)=0$ and $W(F_2^*)=0$,
so  $F_1\in c_4A^+(\Gamma_0(3)^{\psi})$ 
and $F_2\in e_3A^+(\Gamma_0(3)^{\psi})$, and 
$h$ is in the space spanned by 
\[
c_4\{a_1,b_3,e_3\},\quad e_3\{a_1,b_3,e_3\},
\quad \{a_1,e_3,c_4\}, \quad \{b_3,e_3,c_4\}
\]
over $A^+(\Gamma_0(3)^{\psi})$. By \eqref{rel3}, 
we have 
\[
4\c_4\{a_1,b_3,e_3\}+a_1\{b_3,e_3,c_4\}
+3b_3\{e_3,c_4,a_1\}+3e_3\{c_4,a_1,b_3\}=0,
\]
so the generator $e_3\{a_1,b_3,c_4\}$ is redundant.
We denote the inverse image in $J_{*,1}(\Gamma_0(3)^{\psi})$ of 
$(0,c_4\{a_1,b_3,e_3\})$ by $\Phi_{12}$. 
Then $J_{*,1}^-(\Gamma_0(3)^{\psi})$ are generated by 
$\Phi_{14}$, $\Phi_9$, $\Phi_{11}$, $\Phi_{12}$ over 
$A^+(\Gamma_0(3)^{\psi})$. It is obvious that these are
free generators.

\begin{proposition}
$J_{*,1}^{-}(\Gamma_0(3)^{\psi})$ is 
spanned over $A^+(\Gamma_0(3)^{\psi})$ 
by $\Phi_{14}$, $\Phi_{9}$, $\Phi_{11}$, $\Phi_{12}$.
These are free generators. Dimensions are given by 
\[
\sum_{k=1}^{\infty}\dim J_{k,1}^{-}(\Gamma_0(3)^{\psi})t^k
=
\frac{t^9+t^{11}+t^{12}+t^{14}}
{(1-t)(1-t^3)^2(1-t^4)}.
\]
\end{proposition}

\section{Siegel modular forms of level $4$}
We have two cases here. One is the space of 
Jacobi forms of $\Gamma_0(4)^{\psi}$ and 
the other of $\Gamma_0^0(2)^{\psi}$. 
We loosely call them the case of level $4$ and sometimes write 
$\Gamma=\Gamma_0(4)^{\psi}$ or $\Gamma=\Gamma_0^0(2)^{\psi}$.
These two groups are conjugate in $Sp(2,\Q)$ but  
the semi-direct products of $\Gamma$ and the Heisenberg part
$H(\Z)$ are not conjugate, so the space of Jacobi forms are 
different. 

\subsection{Graded rings of level $4$}
First we review the structure of the ring of scalar valued 
Siegel modular forms of any weight and vector 
valued Siegel modular forms of $\Gamma_0(4)^{\psi}$ and 
$\Gamma_0^0(2)^{\psi}$ over the ring of scalar valued 
Siegel modular forms of plus part.
These are essentially the same but 
we choose different generators for each that are 
more suitable for calculation of 
Jacobi forms. In both cases we use the same notation
$a_1$, $b_2$, $c_2$, $d_3$ of weight $1$, $2$, $2$, $3$
but definitions are different according to cases.
We use the same notation to economize the description, since 
the structure of the graded rings 
for $\Gamma_0(4)^{\psi}$ and $\Gamma_0^0(2)^{\psi}$ 
can be written completely in the same way.
We put $\Gamma=\Gamma_0(4)^{\psi}$ or 
$\Gamma_0^0(2)^{\psi}$. 

Later we will define Siegel modular forms $a_1$, $b_2$, $c_2$, $d_3$ of $\Gamma
=\Gamma_0(4)^{\psi}$ or $\Gamma_0^0(2)^{\psi}$ of 
weight $1$, $2$, $2$, $3$, respectively,
and for each case, we put 
\[
\chi_{11}=-\frac{1}{2^{18}\cdot 3}
\begin{vmatrix} 
a_1 & 2b_2 & 2c_2 & 3d_3 \\
\p_{11}a_1 & \p_{11}b_2 & \p_{11}c_2 & \p_{11}d_3 \\
\p_{12}a_1 & \p_{12}b_2 & \p_{12}c_2 & \p_{12}d_3 \\
\p_{22}a_1 & \p_{22}b_2 & \p_{22}c_2 & \p_{22}d_3
\end{vmatrix}.
\] 
Of course this depends on a choise of $\Gamma$. 
Instead of $A^{\pm}(\Gamma)$ we write 
$A^I(\Gamma)=A^+(\Gamma)$ and $A^{II}(\Gamma)=A^{-}(\Gamma)$ since $\pm$ might be a bit confusing.
We have 
\begin{theorem}
We have 
\[
A(\Gamma)=A^I(\Gamma)\oplus A^{II}(\Gamma),
\]
where
\[
A^I(\Gamma)=\C[a_1,b_2,c_2,d_3], \qquad 
A^{II}(\Gamma)=\chi_{11}\C[a_1,b_2,c_3,d_3].
\]
The modules $A_{*,2}^I(\Gamma) =\oplus_{k=0}^{\infty}A_{k,2}^{I}(\Gamma)$ and 
$A_{*,2}^{II}(\Gamma)  =\oplus_{k=0}^{\infty}
A_{k,2}^{II}(\Gamma)$ of vector valued modular forms of weight $\det^k\Sym^2$ of type $I$ and $II$ 
over $A=A^I(\Gamma)$ are 
given by 
\begin{align}\label{gamma00vectorI}
A_{*,2}(\Gamma)^{I} &  =
A\{a_1,b_2\}+A \{a_1,c_2\}+ A\{a_1,d_3\}
+A\{b_2,c_2\}+A\{b_2,d_3\}+A\{c_2,d_3\},
\\ A_{*,2}(\Gamma)^{II}
& =
A\{a_1,b_2,c_2\}+A\{a_1,b_2,d_3\}
+A\{a_1,c_2,d_3\}+A\{b_2,c_2,d_3\}. \label{gamma00vectorII}
\end{align}
(See \cite{aokivector}).
The fundamental relations of the above generators of 
the modules are just the usual 
relations in \eqref{rel2} and \eqref{rel3}.
\end{theorem}
\noindent
Here for $\Gamma=\Gamma_0(4)^{\psi}$
we put 
\begin{align*}
a_1 & =\theta_{0000}^2+\theta_{0001}^2+\theta_{0010}^2+\theta_{0011}^2, \\
b_2 & = \theta_{0000}^4+\theta_{0001}^4+\theta_{0010}^4+\theta_{0011}^4,\\
c_2 & = \theta_{0000}\theta_{0001}\theta_{0010}\theta_{0011}, \\
d_3 & = \theta_{0000}^6+\theta_{0001}^6+\theta_{0010}^6+\theta_{0011}^6.
\end{align*}
For $\Gamma=\Gamma_0^0(2)^{\psi}$ we put 
\begin{align*}
a_1 & = \theta_{0000}^2,\\
b_2 & = \theta_{0000}^4+\theta_{0001}^4+\theta_{0010}^4
+\theta_{0011}^4, \\
c_2 & = \theta_{0000}^4+\theta_{0100}^4+\theta_{1000}^4+\theta_{1100}^4,\\
d_3 & = (\theta_{0001}\theta_{0010}\theta_{0011})^2.
\end{align*}

These generators for each $\Gamma$ are slightly different from those in 
\cite{ibuaoki, aokivector}, 
but we chose convenient generators for our purpose in this paper. 
The proof that they are generators is similar to the one in \cite{ibuaoki} and omitted here.

From next section, we prove our main theorem for 
two level $4$ cases.

\section{Proof for $\Gamma_0(4)^{\psi}$}
First we describe the decomposition of $Sp(2,\Z)$-orbits of diagonals under $\Gamma_0(4)$.

\begin{lemma}\label{gamma04orbit}
We have 
\[
\Gamma_0(4)\backslash Sp(2,\Z)/\Gamma_0=
\{1_4, M_1, M_1^2\},
\]
where we put 
\[
M_1=
\begin{pmatrix} 1_2 & 0 \\
S_0  & 1_2 
\end{pmatrix}.
\]
\end{lemma}
Of course we have $M_1^2=\begin{pmatrix} 1_2 &  0 \\ 2S_0 & 1_2 \end{pmatrix}$.
\begin{proof}
For any $2\times 2$ symmetric matrix $S$, we put 
\[
u(S)=\begin{pmatrix} 1_2 & 0_2 \\ S & 1_2 \end{pmatrix}.
\]
We have 
\[
\Gamma_0(4)\backslash \Gamma_0(2)
=
\{u(2S); S=\,^{t}S\in M_2(\Z), S\bmod 2\Z\},
\]
where $S\bmod 2\Z$ means $S$ runs over representatives
modulo $2$.
Indeed, if $g=\begin{pmatrix} A & B \\ 2C & D \end{pmatrix}
\in \Gamma_0(2)$, then 
\[
g \times u(2S)^{-1}=\begin{pmatrix} A -2BS & B \\ 2(C-DS) & D
\end{pmatrix}.
\]
Since $\det(D)$ is odd and $D^{-1}C$ is symmetric, 
we may choose $S$ such that $2(C-DS)\equiv 0\bmod 4$, so we have $g\in \Gamma_0(4)u(2S)$. 
Now we already know that 
\[
\Gamma_0(2)\backslash Sp(2,\Z)/\Gamma_0=
\{1_4,M_1\},
\]
so representatives in the present case can be taken in 
$u(2S)$ or $u(2S)M_1$. Here $u(2S)$ and $u(2S)M_1$ are never equivalent since they are not equivalent for 
$\Gamma_0(2)$. We assume that $u(2S_1)gu(-2S_2)\in \Gamma_0(4)$ 
for some $g \in \Gamma_0$ and $2\times 2$ symmetric matrices $S_1$, $S_2$. We write $S_i=\begin{pmatrix}
x_i & y_i \\ y_i & z_i \end{pmatrix}$ and 
$g=\iota\left(\begin{pmatrix} p_1 & q_1 \\ r_1 & s_1\end{pmatrix},
\begin{pmatrix} p_2 & q_2 \\ r_2 & s_2 \end{pmatrix}\right)
\in Sp(2,\Z)$ for $\begin{pmatrix} p_i & q_i \\ r_i & s_i \end{pmatrix}\in 
SL_2(\Z)$.
Seeing $(3,1)$ and $(4,2)$ component of 
$u(2S_1)gu(-2S_2)\in \Gamma_0(4)$, we see that 
$r_1$, $r_2$ are even. So $p_i$, $s_i$ should be odd.
Then we see $y_1\equiv y_2 \bmod 2$ if 
$u(2S_1)gu(-2S_2)\in \Gamma_0(4)$. When $y_1=y_2=0$ 
or $y_1=y_2=1$, we see that 
$u(2S_1)$ and $u(2S_2)$ are equvalent 
by taking $q_1=q_2=0$ and $p_i=s_i=1$ and even 
$r_i$ suitably.
When $(y_1,y_2)=(0,1)$ or $(1,0)$, then seeing 
the $(3,2)$ and $(4,1)$ components, we see that 
$u(2S_1)$ and $u(2S_2)$ are not equivalent.
So we have different representatives for 
$S=0$ and $S=S_0$. 
Next we consider $u(2S_1)M_1gM_1^{-1}u(-2S_2)$.
We write $2S_i+S=\begin{pmatrix} x_i & y_i\\
y_i & z_i \end{pmatrix}$. Then $y_1\equiv y_2 \equiv 1 
\bmod 2$.
Taking $q_1=q_2=0$, the left lower block is given by
\[
\begin{pmatrix} x_1p_1-s_1x_2 + r_1 & y_1p_2-s_1y_2 \\
y_1p_1-s_2y_2 & z_1p_2-s_2z_2+r_2 \end{pmatrix}.
\]
By taking $r_i$ suitably, we can make both $(1,1)$ and $(2,2)$ components to be divisible by $4$. 
Now if $y_1\equiv y_2 \bmod 4$, then taking 
$p_1=p_2=s_1=s_2=1$, this matrix is divisible by $4$. 
If $y_1\equiv y_2+2\bmod 4$, then 
taking $p_1=s_1=1$ and $p_2=s_2=-1$, we have 
$y_1p_2-s_1y_2 = -(y_1+y_2)=-2y_2-2 \equiv 0 \bmod 4$
and $y_1p_1-s_2y_2= y_1+y_2\equiv 2y_2+2 \equiv 0 
\bmod 4$ since $y_i$ is odd. 
Hence all these are equivalent to $M_1$ itself.
This proves the claim.
\end{proof}
By Proposition \ref{transformation}, we see 
\[
\begin{array}{|l|c|} \hline
\text{\hspace{4ex}Transformations}    &  \text{Witt images} \\ \hline 
a_1   = \theta_{0000}^2+\theta_{0001}^2+\theta_{0010}^2+
\theta_{0011}^2  & (A^2+B^2)(a^2+b^2) \\
b_2   = \theta_{0000}^4+\theta_{0001}^4+\theta_{0010}^4+\theta_{0011}^4 & (A^4+B^4)(a^4+b^4) \\
c_2  = \theta_{0000}\theta_{0001}\theta_{0010}\theta_{0011} & A^2B^2a^2b^2 \\
d_3  = \theta_{0000}^6+\theta_{0001}^6+\theta_{0010}^6+\theta_{0011}^6  &  (A^6+B^6)(a^6+b^6) \\ \hline 
a_1|[M_1]  = \theta_{0000}^2+\theta_{1001}^2+\theta_{0110}^2-\theta_{1111}^2 & A^2a^2+B^2c^2+C^2b^2 \\
b_2|[M_1] = \theta_{0000}^4+\theta_{1001}^4+\theta_{0110}^4+\theta_{1111}^4 & A^4a^4+B^4c^4+C^4b^4 \\
c_2|[M_1] = e(-1/4)\theta_{0000}\theta_{1001}\theta_{0110}\theta_{1111} & 0 \\
d_3|[M_1] = \theta_{0000}^6+\theta_{1001}^6+\theta_{0110}^6-\theta_{1111}^6 & A^6a^6+B^6c^6+C^6b^6 \\ \hline 
a_1|[M_1^2] =  a_1 &  (A^2+B^2)(a^2+b^2) \\
b_2|[M_1^2] =  b_2 & (A^4+B^4)(a^4+b^4) \\
c_2|[M_1^2] =  -c_2 & -A^2B^2a^2b^2 \\
d_3|[M_1^2] = d_3 & (A^6+B^6)(a^6+b^6) \\ \hline 
\end{array}
\]
We put
\begin{align*}
f_3 & =d_3+2^{-1}a_1(a_1^2-3b_2-6c_2), \\
g_3 & = d_3+2^{-1}a_1(a_1^2-3b_2+6c_2).
\end{align*}
We can show that $F \in \C[a_1,b_2,c_2,d_3]$ with $W(F)=0$ is 
generated by $f_3$. Since $W(K_6)=0$, we have $K_6=f_3g$ for 
some $g$, but since $g_3=f_3|_3[M_1^2]$ and $K_6|[M_1^2]=K_6$, 
we have $g_3(g|_3[M_1^2])=f_3g$.  
Hence $g|_3[M_1^2]$ is divisiblle by $f_3$, so $K_6$ is a constant 
times $f_3g_3$. Comparing the Fourier coefficients, we see    
\[
f_3g_3=-36864 K_6.
\]
So we have 
\[
\chi_{10}=-c_2^2 f_3g_3/36864.
\]
We also see that $Ker(W) \cap \C[a_1|[M_1^2],b_2|[M_1^2],c_2|[M_1^2],d_3|[M_1^2])$ is generated by $g_2$ and 
$Ker(W)\cap \C[a_1|[M_1],b_2|[M_1],c_2|[M_1],d_3|[M_1]]$
is generated by $c_2$. 
The orbit of $Sp(2,\Z)$ of diagonals are divided by 
$\{f_3=0\}$, $\{g_3=0\}$, and $\{c_2=0\}$. 

\subsection{Proof for type I}
We have 
\begin{align*}
W(\p_{12}(c_2|[M_1])& =e(-1/4)W(\theta_{0000}\theta_{1001}\theta_{0110}
\p_{12}\theta_{1111})
\\ & =\frac{\sqrt{-1}}{4}A^2B^2C^2a^2b^2c^2\neq 0.
\end{align*}
So there is no 
$\Phi\in J_{2.1}^{I}(\Gamma_0(4)^{\psi})$ such that $\xi(\Phi)=(c_2,*)$,
since $A_{2,2}(\Gamma_0(4)^{\psi})=0$. 
On the other hand, we have 
$W(\p_{12}(f|[M]))=0$ for $f=a_1$, $b_2$, $d_3$ for $M=1_4$, $M_1$, $M_1^2$ 
and $W(\p_{12}(c_2|[M_1^2]))=W(\p_{12}c_2)=0$
and $W(\p_{12}(c_2^2|[M_1])=0$. So we have 
$(a_1,0)$, $(b_2,0)$, $(c_2^2,0)$, $(d_3,0)\in 
\xi(J_{*,1}^{I}(\Gamma_0(4)^{\psi}))$.
We will show that these four elements span 
$J_{*,1}^{I}(\Gamma_0(4)^{\psi})$ over $A^I(\Gamma_0(4)^{\psi})$ 
freely. These four elements are linearly independent over 
$A^I(\Gamma_0(4)^{\psi})$ by taking $G_1=c_1^2$, $G_2=a_1$,
$G_3=b_2$, $G_4=d_3$ in Lemma \ref{freebasis} and \cite{aokivector}.
 
Next we show these span the space 
$J_{*,1}^{I}(\Gamma_0(4)^{\psi})$. 
If we put 
\[
\fM=A^I(\Gamma_0(4)^{\psi})a_1+A^I(\Gamma_0(4)^{\psi})b_2
+A^I(\Gamma_0(4)^{\psi})c_2^2+A^I(\Gamma_0(4)^{\psi})d_3,
\]
then by using \eqref{moduleaction} similarly as in the case 
of level $3$, we can show that $\xi_0(J_{*,1}^{I}(\Gamma_0(4)^{\psi})
=\fM$. So subtracting some element in $\fM$ 
from a form $\Phi\in J_{*,1}^{I}(\Gamma_0(4)^{\psi})$, 
we may assume that
$\xi(\Phi)=(0,h)$ for $h\in A_{k,2}^{I}(\Gamma_0(4)^{\psi})$. 
Here we write $W(h_{11})=W_{11}(h)$ where $h_{11}$ is 
the coefficient of $u_1u_2$ of $h$. We write 
\[
\cM=\{h\in A_{*,2}^{I}(\Gamma_0(4)^{\psi}); W_{11}((h|[M])=0 
\text{ for all $M=1_4$, $M_1$, $M_1^2$}\}.
\]
Then $(0,h)\in {\rm Image}(\xi)$ if and only if $h\in \cM$.
We will determine $\cM$. 
We have 
$W(\p_{12}a_1|[M])=W(\p_{12}b_2|[M])=W(\p_{12}d_3|[M])=0$ 
for $M=1_4$, $M_1$, $M_1^2$ and 
$W(\p_{12}c_2|[M])=0$ for $M=1_4$, $M_1^2$, but 
$W(\p_{12}(c_2|[M_1])\neq 0$. So we have $\{a_1,b_2\}$,
$\{a_1,d_3\}$, $\{b_2,d_3\}\in \cM$. 
Now we assume that 
\[
h=F_1\{c_2,a_1\}+F_2\{c_2,b_2\}+F_3\{c_2,d_3\}\in \cM.
\]
We have $W_{11}(h|[M])=0$ for $M=1_4$ and $M_1^2$, so the problem 
is the condition for $M_1$. For the sake of simplicity, we write 
$f|[M_1]=f^*$ for any $f \in A^{I}(\Gamma_0(4)^{\psi})$.
Then for $f=a_1$, $b_2$, $d_3$, we have 
\[
W_{11}(\{c_2^*,f^*\})=W(2c_2^*\p_{12}f\*)-kW(f^*)W(\p_{12}c_2^*)
=-kW(f^*)W(\p_{12}c_2^*),
\]
where $k$ is the weight of $f$. So the condition $W_{11}(h|[M_1])=0$
is equivalent to the condition
\[
W((F_1a_1+2F_2b_2+3d_3F_3)|[M_1])=0.
\]
But the kernel of $W(f^*)$ is spanned by $c_2$, so we have 
\[
F_1a_1+2F_2b_2+3d_3F_3=c_2G
\]
for some $G$. 
Since $a_1$, $b_2$, $c_2$, $d_3$ are algebraically independent, 
we may write uniquely as 
\[
F_i=H_i(a_1,b_2,d_3)+c_2G_i(a_1,b_2,c_2,d_3)
\]
for $i=1$, $2$, $3$, and substituting $c_2=0$, we  have 
\begin{equation}\label{shikiH}
a_1H_1+2b_2H_2+3d_3H_3=0.
\end{equation}
Here we put 
\[
H_i(a_1,b_2,d_3)=P_i(a_1,b_2)+d_3Q_i(a_1,b_2,d_3)
\]
for $i=1$, $2$. Then seeing \eqref{shikiH}, we have 
\begin{align}
& a_1P_1(a_1,b_2)+2b_2P_2(a_1,b_2)=0, \label{shikiP}
\\ & a_1Q_1(a_1,b_2,d_3)+2b_2Q_2(a_1,b_2,d_3)+3H_3=0.
\label{shikiQ}
\end{align}
Here \eqref{shikiP} means that we have 
$P_1(a_1,b_2)=-2b_2R(a_1,b_2)$, $P_2(a_1,b_2)=a_1R(a_1,b_2)$
for some polynomial $R$. 
So going back to the first and erasing $P_i$ and $H_3$ by \eqref{shikiQ} and \eqref{shikiP}, we have 
\begin{align*}
& F_1\{c_2,a_1\}+F_2\{c_2,b_2\}+F_3\{c_2,d_3\}
\\ = & G_1c_2\{c_2,a_1\}+G_2c_2\{c_2,b_2\}+G_3c_2\{c_2,d_3\}
 + R(-2b_2\{c_2,a_1\}+a_1\{c_2,b_2\})
\\ & + \frac{Q_1}{3}(3d_3\{c_2,a_1\}-a_1\{c_2,d_3\})
 + \frac{Q_2}{3}(3d_3\{c_2,b_2\}-2b_2\{c_2,d_3\}).
\end{align*}
Here we have 
\begin{align*}
-2b_2\{c_2,a_1\}+a_1\{c_2,b_2\} & = 2c_2\{a_1,b_2\}, \\
3d_3\{c_2,a_1\}-a_1\{c_2,d_3\} & = 2c_2\{d_3,a_1\}, \\
3d_3\{c_2,b_2\}=2b_2\{c_2,d_3\} & = 2c_2\{d_3,b_2\}.
\end{align*}
So we have 
\[
\cM=Span(\{a_1,b_2\}, \{a_1,d_3\}, \{b_2,d_3\},
c_2\{c_2,a_1\}, c_2\{c_2,b_2\}, c_2\{c_2,d_3\}).
\]
Here we note that $\{c_2^2,f\}=2c_2\{c_2,f\}$. 
By \eqref{rel1}, we see that 
the module 
\[
A^{I}(\Gamma_0(4)^{\psi})(a_1,0)
+A^I(\Gamma_0(4)^{\psi})(b_2,0)
+A^I(\Gamma_0(4)^{\psi})(c_2^2,0)
+A^I(\Gamma_0(4)^{\psi})(d_3,0)
\]
contains $(0,\cM)$.  So we completed a proof of the following theorem.
\begin{theorem}
For $k=1$, $2$, $3$, $4$, we denote by $\Phi_k$ 
the element of $J_{k,1}^{I}(\Gamma_0(4)^{\psi})$
such that $\xi(\Phi_1)=(a_1,0)$, $\xi(\Phi_2)=(b_2,0)$, 
$\xi(\Phi_3)=(d_3,0)$ and $\xi(\Phi_4)=(c_2^2,0)$. 
Then we have 
\[
J_{*,1}^{I}(\Gamma_0(4)^{\psi})
=A^I(\Gamma_0(4)^{\psi})\Phi_1\oplus 
A^I(\Gamma_0(4)^{\psi})\Phi_2\oplus 
A^I(\Gamma_0(4)^{\psi})\Phi_3\oplus 
A^I(\Gamma_0(4)^{\psi})\Phi_4.
\]
\end{theorem}

\subsection{Proof for type II}
First we show that there is some $h\in A_{11,2}(\Gamma_0(4)^{\psi})$ 
such that $(\chi_{11},h)\in {\rm Image}(\xi)$. 
To see this, we need $W(\p_{12}\chi_{11})$ for various orbits of 
diagonals.
For the sake of simplicity, we put $\X_{11}=-2^{18}\cdot 3\chi_{11}$.  
In the definition of $\chi_{11}$, we can replace $d_3$ by $f_3$ or $g_3$.
We have $W(\p_{12}f_3)=W(\p_{12}\p_{ii}f_3)=0$ for $i=1$. $2$, and 
since $g_3|[M_1^2]=f_3$, we have 
\begin{align*}
W(\p_{12}X_{11}) & =-W(\p_{12}^2f_3)W(\{a_2,b_2,c_2\}_{11})/2, \\
W(\p_{12}(X_{11}|[M_1^2]) & = -W(\p_{12}^2f_3)W(\{a_2|[M_1^2],b_2|[M_1^2],c_2|[M_1^2]\}_{11})/2.
\end{align*}
Now we show that $W(\p_{12}(\X_{11}|[M_1]))=0$. 
We have $W(c_2|[M_1])=W(\p_{12}(c_2|[M_1])=W(\p_{12}\p_{ii}(c_2|[M_1])=0$
for $i=1$, $2$. Since $c_2|[M_1]$ is an odd function with respect to $\tau_{12}$, we have 
$W(\p_{12}^2(c_2|[M_1]))=0$. If we write $A_i$ the derivative by $\p_{12}$ 
of $i$-th row of $\X_{11}$, then if $i\neq 3$, then 
the third row ant the $i$-the row of $W(A_i)$ are
both vectors $(0,0,*,0)$ for some $*$, so $W(A_i)=0$. 
The third row of $W(A_3)$ is the zero vector, so again $W(A_3)=0$
and $W(\p_{12}(\X_{11}|[M_1]))=0$. 
Now we would like to find a form $F \in A_{5}^I(\Gamma_0(4)^{\psi})$ 
such that $W(F|[M])=W(\p_{12}^2 (f_3|[M]))$ for $M=1_4$ and $M_1^2$. 
For that, we must calculate $W(\p_{12}^2f_3)$. 
Since $f_3g_3=-36864 K_6$, we have 
\begin{align*}
-36863W(\p_{12}^2K_6) & =
W(\p_{12}^2f_3)W(g_3)+2W(\p_{12}f_3\p_{12}g_3)+W(f_3)W(\p_{12}^2g_3)
\\ & =W(\p_{12}^2f_3)W(g_3).
\end{align*}
Here 
\begin{align*}
W(4096\p_{12}^2K_6)& =2W(\theta_{0100}\theta_{0110}\theta_{1000}\theta_{1001}\theta_{1100})^2)
W((\p_{12}\theta_{1111})^2)
\\ & = \frac{1}{8}A^4B^4C^{8}a^4b^4c^{8}.
\end{align*}
On the other hand, we have
\[
W(g_3)=6A^2a^2B^2b^2(A^2+B^2)(a^2+b^2), 
\]
so we have 
\[
W(\p_{12}^2f_3)=-\frac{3}{16}A^2B^2a^2b^2(A^2+B^2)(A^2-B^2)^2(a^2+b^2)(a^2-b^2)^2.
\]
On the other hand, if we put 
\[
F_0
=-\frac{1}{6}a_1^5+\frac{5}{6}a_1^3b_2-a_1b_2^2-\frac{1}{3}a_1^2d_3+8a_1c_2^2+\frac{2}{3}
b_2d_3,
\]
then we have 
\[
W(\p_{12}^2f_3)=-\frac{3}{16}W(F_0).
\]
It is obvious that $F_0|[M_1^2]=F_0$ and
 $W(\p_{12}^2(g_3|[M_1^2])=-\frac{3}{16}W(F_0|[M_1^2])$.
Now put 
\[
H=\frac{3(2\pi i)^2}{16\cdot 11}F_0\{a_1,b_2,c_2\}.
\]
Then we have $W((H|[M_1])_{11})=0$ since 
$W(c_2|[M_1])=W(\p_{ii}(c_2|[M_1])=0$ for $i=1$, $2$, and 
for any $M\in Sp(2,\Z)$, we have 
\[
W\biggl(\frac{(2\pi i)^2}{11}\p_{12}(\X_{11}|[M])+H_{11}^M\biggr)=0.
\]
So $(\X_{11},H)=\xi(\Phi_{11})$ for some $\Phi\in J_{11,1}^{II}(\Gamma_0(4)^{\psi})$.
So if $(f,h)\in \xi(J_{*,1}^{II}(\Gamma_0(4)^{\psi})$, then 
we can adjust $f=0$ by subtracting some element in 
$A^{I}(\Gamma_0(4)^{\psi})\Phi_{11}$.
So now we must determine elements in $J_{*,1}^{II}(\Gamma_0(4))$ 
whose images by $\xi$ are $(0,h)$.
This means that we must determine $h\in A_{*,2}^{II}(\Gamma_0(4)^{\psi})$
such that $W(h^{[M]}_{11})=0$ for all $M=1_4$, $M_1$, $M_1^2$. 
First we determine such $h$ for each $M$. 
We put 
\begin{align*}
V_0 & =\{h\in A_{*,2}^{II}(\Gamma_0(4)^{\psi}); W(h_{11})=0\},\\
V_1 & = \{h\in A_{*,2}^{II}(\Gamma_0(4)^{\psi}); W((h|[M_1])_{11})=0\}, \\
V_2 & = \{h \in A_{*,2}^{II}(\Gamma_0(4)^{\psi}); W((h|[M_1^2])_{11})=0\}.
\end{align*}
The module $A_{*,2}^{II}(\Gamma_0(4)^{\psi})$ is spanned over $A^I(\Gamma_0(4)^{\psi})$ 
by $\{a_1,b_2,c_2\}$, $\{a_1,b_2,d_3\}$, $\{a_1,c_2,d_3\}$, $\{b_2,c_2,d_3\}$. 
Since $d_3=f_3-2^{-1}a_1(a_1^2-3b_2-6c_2)$, we have 
\begin{align*}
\{a_1,b_2,d_3\} & =\{a_1,b_2,f_3\}+3\{a_1,b_2,c_2\},\\
\{a_1,c_2,d_3\} & = \{a_1,c_2,f_3\}-(3/2)\{a_1,b_2,c_2\}, \\
\{b_2,c_2,d_3\} & = \{b_2,c_2,f_3\}-(9/2)a_1^2\{a_1,b_2,c_2\}.
\end{align*}
So $A_{*,2}^{II}(\Gamma_0(4)^{\psi})$ is also spanned by 
$\{a_1,b_2,c_2\}$, $\{a_1,b_2,f_3\}$, $\{a_1,c_2,f_3\}$, $\{b_2,c_2,f_3\}$. 
Here we have $W(f_3)=W(\p_{ii}f_3)=0$ ($i=1,2$). We also have 
$W(\{a_1,b_2,c_2\}_{11})\neq 0$. So we have 
\[
V_0=Span(f_3\{a_1,b_2,c_2\},\{a_1,b_2,f_3\},\{a_1,c_2,f_3\},\{b_2,c_2,f_3\}),
\]
where Span means the span over $A^{I}(\Gamma_0(4)^{\psi})$.
Completely in the same way, we also have 
\begin{align*}
V_1 & = Span(c_2\{a_1,b_2,g_3\},\{a_1,b_2,c_2\},\{a_1.c_2,g_3\},\{b_2,c_2,g_3\}),  \\ 
V_2 & = Span(g_3\{a_1,b_2,c_2\},\{a_1,b_2,g_3\},\{a_1,c_2,g_3\},\{b_2,c_2,g_3\}).
\end{align*}
Here if we put $W_0=span(\{a_1,c_2,f_3\},\{b_2,c_2,f_3\})$, then we have 
$W_0\subset V_0\cap V_1$. Since
\begin{equation}\label{triplerelation}
3f_3\{a_1,b_2,c_2\}=-a_1\{b_2,c_2,f_3\}-b_2\{a_1,c_2,f_3\}-c_2\{a_1,b_2,f_3\}.
\end{equation}
We have 
\[
V_0=Span(W_0,\{a_1,b_2,f_3\}), \qquad V_1=Span(W_0,\{a_1,b_2,c_2\}).
\]
Now we consider a relation 
\[
F_1\{a_1,b_2,f_3\}=F_2\{a_1,b_2,c_2\}+F_3\{a_1,c_2,f_3\}+F_4\{b_2,c_2,f_3\}.
\]
Since the fundamental relation of these four generators are just 
\eqref{triplerelation} (See \cite{aokivector}), we see that  
$F_1$ is divisible by $c_2$ and $F_2$ is divisible by $f_3$. So 
we have 
\[
V_0\cap V_1
=Span(f_3\{a_1,b_2,c_2\},\{a_1,c_2,f_3\},\{b_2,c_2,f_3\}).
\]
In the same way, we have 
\[
V_1\cap V_2 = Span(g_3\{a_1,b_2,c_2\},\{a_1,c_2,g_3\},\{b_2,c_2,g_3\}).
\]
Since $g_3=f_3+6a_1c_2$, elements of 
$V_1\cap V_2$ is written as 
\begin{multline*}
F_1(f_3+6a_1c_2)\{a_1,b_2,c_2\}
+F_2\{a_1,c_2,g_3\}+F_3\{b_2,c_2,g_3\}
\\ = 
(F_1(f_3+6a_1c_2)+6c_2F_3)\{a_1,b_2,c_2\}+F_2\{a_1,c_2,f_3\}+F_3\{b_2,c_2,f_3\}
\end{multline*}
for some elements $F_i\in A^I(\Gamma_0(4)^{\psi})$. 
(We note that $\{a_1,c_2,f_3\}=\{a_1,c_2,g_3\}$.) 
If this is in $V_0\cap V_1$, then since $\{a_1,b_2,c_2\}$, $\{a_1,c_2,f_3\}$,
 $\{b_2,c_2,f_3\}$ are linearly independent over $A^{I}(\Gamma_0(4)^{\psi})$, 
we have $a_1F_1+F_3=f_3G$ for some $G \in A^{I}(\Gamma_0(4)^{\psi})$. 
In other words, elements of $V_0\cap V_1\cap V_2$ is given by 
\[
F_1(g_3\{a_1,b_2,c_2\}-a_1\{b_2,c_2,f_3\})+F_2\{a_1,c_2,g_3\}
+G_3f_3\{b_2,c_2,g_3\}.
\]
This means that 
\[
V_0\cap V_1 \cap V_2=Span(g_3\{a_1,b_2,c_2\}-a_1\{b_2,c_2,f_3\},
\{a_1,c_2,g_3\},f_3\{b_2,c_2,g_3\}). 
\]

So we prove 
\begin{theorem}
$J_{*,1}^{II}(\Gamma_0(4)^{\psi})$ 
is spaned by four free generators over $A^{I}(\Gamma_0(4)^{\psi})$
of weight $11$, $9$, $7$, $11$, respectively,  whose images by $\xi$ are 
\begin{align*}
& \bigl(X_{11},\frac{3(2\pi i)^2}{16\cdot 11}F_0\{a_1,b_2,c_2\}\bigr),
\quad (0,g_3\{a_1,b_2,c_2\}-a_1\{b_2,c_2,f_3\}),
\\ & (0,\{a_1,c_2,g_3\}), \quad \quad (0,f_3\{b_2,c_2,g_3\}).
\end{align*}
\end{theorem}

\section{Proof for $\Gamma_0^0(2)^{\psi}$}\label{prooflevel2}
\begin{lemma}\label{gamma002orbit}
We have 
\[
\Gamma_0^0(2)\backslash Sp(2,\Z)/\Gamma_0=
\{1_4, M_1,M_2,M_3\}
\]
where we put 
\[
M_1=\begin{pmatrix} 1_2 & 0 \\ S_0 & 1_2 \end{pmatrix},
\quad 
M_2=
\begin{pmatrix} 1_2 & S_0  \\
0 & 1_2 \end{pmatrix},
\quad 
M_3=
\begin{pmatrix} 1_2 & -1_2-S_0  \\
1_2 & -S_0 
\end{pmatrix}, \quad S_0=\begin{pmatrix} 0 & 1 \\ 1 & 0 
\end{pmatrix}.
\]
\end{lemma}

\begin{proof}
For a $2\times 2$ symmetric matrix $S$, we put 
\[
v(S)=\begin{pmatrix} 1_2 & S \\ 0_2 & 1_2 \end{pmatrix}.
\]
Then we have 
\[
\Gamma_0^0(2)\backslash \Gamma_0(2)
=
\{v(S); S=\,^{t}S\in M_2(\Z), S \bmod 2\}.
\]
So the representatives of the double coset are
taken among $v(S)$ and $v(S)M_1$. 
Here we can show that representatives are given by 
$v(S) $ and $v(S)M_1$ for $S=0$ and 
$S=S_0$
by the same sort of calculation as in the proof of 
Lemma \ref{gamma04orbit}. We omit the details here.
\end{proof}

By Proposition \ref{transformation}, we can show the 
following relations.
\[
\begin{array}{|l|c|} \hline
\text{  \hspace{4ex}Transformations} & \text{Witt images} \\ \hline
a_1 = \theta_{0000}^2  & A^2 a^2 \\
b_2 = \theta_{0000}^4+\theta_{0001}^4+\theta_{0010}^4+\theta_{0011}^4 & (A^4+B^4)(a^4+b^4) \\
c_2 = \theta_{0000}^4+\theta_{0100}^4+\theta_{1000}^4+\theta_{1100}^4 & (A^4+C^4)(a^4+c^4) \\
d_3 = (\theta_{0001}\theta_{0010}\theta_{0011})^2 & A^2B^4a^2b^4 
\\ \hline 
a_1|[M_1] = a_1 & A^2a^2 \\
b_2|[M_1]  =  \theta_{0000}^4+\theta_{1001}^4+\theta_{0110}^4+\theta_{1111}^4 
& A^4a^4+B^4c^4+C^4b^4 \\
c_2|[M_1]  =  \theta_{0000}^4+\theta_{0100}^4 + \theta_{1000}^4+
\theta_{1100}^4 & (A^4+C^4)(a^4+c^4) \\
d_3|[M_1]  =   -(\theta_{0110}\theta_{1001}\theta_{1111})^2 & 0 \\ \hline 
a_1|[M_2]  =  a_1  & A^2a^2 \\
b_2|[M_2]  =  \theta_{0000}^4+\theta_{0001}^4+\theta_{0010}^4+\theta_{0011}^4 & (A^4+B^4)(a^4+b^4) \\
c_2|[M_2] = \theta_{0000}^4+\theta_{0110}^4+\theta_{1001}^4+\theta_{1111}^4 & A^4a^4+B^4c^4+C^4b^4 \\
d_3|[M_2]  = d_3 &  A^2B^4a^2b^4
\\ \hline 
a_1|[M_3]  =  \theta_{1111}^2 & 0 \\
b_2|[M_3]  =  \theta_{1111}^4+\theta_{1001}^4+\theta_{0110}^4
+\theta_{0000}^4 & A^4a^4+B^4c^4+C^4b^4 \\
c_2|[M_3]  =  \theta_{1111}^4-\theta_{1000}^4-\theta_{0100}^4+\theta_{0011}^4 & -A^4c^4-C^4a^4+B^4b^4 \\
d_3|[M_3]  =   -(\theta_{0000}\theta_{0110}\theta_{1001})^2 &  -A^2B^2C^2a^2b^2c^2 \\ \hline 
\end{array} 
\]
Now we explain that each orbit is the zero of some 
modular form.
We put 
\begin{align*}
f_3 & =(\theta_{0110}\theta_{1001}\theta_{1111})^2,\\
g_3 & = (\theta_{0100}\theta_{1000}\theta_{1100})^2.
\end{align*}
By the theta transformation formula in \cite{igusagraded1} p. 227, 
we can show that 
$f_3$, $g_3 \in A_3(\Gamma_0^0(2), \psi)$.
By definition, we have 
\[
K_6=f_{3}g_{3}/4096,
\]
where $K_6$ is the generator of $A(\Gamma_0(2))$ given 
in section \ref{prooflevel1and2}.
We can show that any element $F \in \C[a_1,b_2,c_2,d_3]$
such that $W(F)=0$ is a constant times of 
$(6a_1^3-2a_1b_2-a_1c_2+3d_3)$, we see
\[
f_3  =-d_3-2a_1^3+\frac{2}{3}a_1b_2+\frac{1}{3}a_1c_2.
\]
by comparing the Fourier coefficients. By comparing $W(g_3)$ and
$W(a_1b_2-a_1c_2-3d_3)$, we see that 
$g_3=-(a_1b_2-a_1c_2-3d_3)/3+F_0$, 
where $W(F_0)=0$. Since $F_0$ is a constant multiple of $f_3$, comparing Fourier coefficients, we see that $F_0=0$. 
So we have 
\begin{equation}
g_3  =d_3-\frac{1}{3}a_1b_2+\frac{1}{3}a_1c_2.
\end{equation}
By definition of $Y_4$ in section \ref{prooflevel1and2}, we also have
\[
Y_4=a_1d_3.
\]
So we have 
\[
\chi_{10}=\frac{1}{4096}a_1d_3f_3g_3.
\]
The diagonal orbits of $\chi_{10}=0$ is given by 
the union of 
$D_0=\{f_3=0\}$, $D_1=\{g_3=0\}$, $D_3=\{a_1=0\}$, and $D_4=\{d_3=0\}$. 

\subsection{Proof for type I}
\begin{theorem}
$J_{*,1}^{I}(\Gamma_0^0(2)^{\psi})$ is spanned by
a free basis $(a_1,0)$, $(b_2,0)$, $(c_2,0)$, $(d_3,0)$ 
 over $A^I(\Gamma_0^0(2)^{\psi})=\C[a_1,b_2,c_2,d_3]$. 
\end{theorem}

\begin{proof}
For any characteristics $m$, we have 
$W(\p_{12}\theta_{m}^2)=0$, so 
for any $f=a_1$, $b_2$, $c_2$, $d_3$, and 
any $M=1_4$, $M_1$, $M_2$, $M_3$, we have 
\[
W(\p_{12}(f|[M])=0.
\]
This means that $(a_1,0)$, $(b_2,0)$, $(c_2,0)$, $(d_3,0)
\in {\rm Image}(\xi)$. 
By Lemma \ref{freebasis}, these are linearly independent over 
$A^I(\Gamma_0^0(2)^{\psi})$. For any $\xi(\Phi)=(f,h)$ for 
$\Phi\in J_{*,1}^{I}(\Gamma_0^0(2)^{\psi})$, subtracting a linear
combination of $(a_1,0)$, $(b_2,0)$, $(c_2,0)$, $(d_3,0)$, 
we can make $f$ to be $0$ and $(f,h)=(0,h)$. By 
\eqref{rel1}, we can take $h$ to be any of 
$\{a_1,b_2\}$, $\{a_1,c_2\}$, $\{a_1,d_3\}$, $\{b_2,c_2\}$, 
$\{b_2,d_3\}$, $\{c_2,d_3\}$. These span $A_{*,2}^{I}(\Gamma_0^0(2)^{\psi})$ over $A^I(\Gamma_0^0(2)^{\psi})$, so 
$h$ can be any element in $A_{*,2}^I(\Gamma_0^0(2))$. 
So the above four elements span $J_{*,1}(\Gamma_0^0(2)^{\psi})$.
\end{proof}

\subsection{Proof for type II}
First we show that $(\chi_{11},h)\in {\rm Image}(\xi)$ for some 
$h\in A_{11,2}^{I}(\Gamma_0^0(2)^{\psi})$. 
As before we replace $\chi_{11}$ by $\X_{11}$. 
Similarly as before, we can show that 
\begin{align*}
2W(\p_{12}X_{11}) & = W(\p_{12}^2f_3)W(\{a_1,b_2,c_2\}_{11}),\\
2W(\p_{12}(X_{11}|[M_1])) & = -W(\p_{12}^2(d_3|[M_1])W(\{a_1|[M_1],b_2|[M_1],c_2|[M_1]\}_{11}), \\
2W(\p_{12}(X_{11}|[M_2])) & = -W(\p_{12}^2(g_3|[M_2])W(\{a_1|[M_2],b_2|[M_2],c_2|[M_2])\}_{11}), \\
2W(\p_{12}(X_{11}|[M_2])) & = W(\p_{12}^2(a_1|[M_3])W(\{b_2|[M_3],c_2|[M_3],d_3|[M_3]\}_{11}).
\end{align*}
On the other hand, we have 
\begin{align*}
W(\p_{12}^2f_3) & =\frac{1}{8}A^2B^4C^4a^2b^4c^4, \\
W(\p_{12}^2(d_3|[M_1]) & = -\frac{1}{8}A^2B^4C^4a^2b^4c^4, \\
W(\p_{12}^2(g_3|[M_2]) & = -\frac{1}{8}A^2B^4C^4a^2b^4c^4, \\
W(\p_{12}^2(a_1|[M_3]) & = \frac{1}{8}A^2B^2C^2a^2b^2c^2. 
\end{align*}
Now for the sake of simplicity, for $H\in A_{*,2}^{II}(\Gamma_0^0(2))$ and $M\in Sp(2,\Z)$, we write $W_{11}(H|[M])=W((H|[M])_{11})$. 
Here we want to find $H\in A_{k,2}^{II}(\Gamma_0^0(2))$ 
such that 
\begin{align}
W_{11}(H) & = -A^2B^4C^4a^2b^4c^4W(\{a_1,b_2,c_2\}_{11}), \label{shiki1}\\
W_{11}(H|[M_1]) & = -A^2B^4C^4a^2b^4c^4W((\{a_1,b_2,c_2\}|[M_1])_{11}), \label{shiki2}\\
W_{11}(H|[M_2]) & = -A^2B^4C^4a^2b^4c^4W((\{a_1,b_2,c_2\}|[M_2])_{11}), \label{shiki3}\\
W_{11}(H|[M_3]) & = -A^2B^2C^2a^2b^2c^2 W((\{b_2,c_2,d_3\}|[M_3])_{11}). \label{shiki4}
\end{align}
We may write 
\begin{equation}\label{H}
H=F_1\{a_1,b_2,c_2\}+F_2\{a_1,b_2,d_3\}+F_3\{a_1,c_2,d_3\}+F_4\{b_2,c_2,d_3\}
\end{equation}
for some $F_i\in A_{*}^{I}(\Gamma_0^0(2)^{\psi})$. 
Since we have 
\begin{align*}
\{a_1,b_2,d_3\} & =-\{a_1,b_2,f_3\}+\frac{1}{3}a_1\{a_1,b_2,c_2\}, \\
\{a_1,c_2,d_3\} & =-\{a_1,c_2,f_3\}-\frac{2}{3}a_1\{a_1,b_2,c_2\}, \\
\{b_2,c_2,d_3\} & = -\{b_2,c_2,f_3\}-6a_1^2\{a_1,b_2,c_2\}+\frac{2}{3}b_2\{a_1,b_2,c_2\}
+\frac{1}{3}c_2\{a_1,b_2,c_2\},
\end{align*}
and $W(\{a_1,b_2,f_3\}_{11})=W(\{a_1,c_2,f_3\}_{11})=W(\{b_2,c_2,f_3\}_{11})=0$, 
a necessary condition so that 
\eqref{shiki1} holds can 
be written as 
\begin{equation}\label{eq1}
F_1+\frac{1}{3}a_1F_2-\frac{2}{3}a_1F_3+\frac{1}{3}(-18a_1^2+2b_2+c_2)F_4
\sim -A^2B^4C^4a^2b^4c^4,
\end{equation}
where $\sim$ means images of the both sides under $W$ are the same.
We also have 
\begin{align*}
\{a_1,b_2,d_3\} & =\{a_1,b_2,g_3\}-\frac{1}{3}a_1\{a_1,b_2,c_2\}, \\
\{a_1,c_2,d_3\} & =\{a_1,c_2,g_3\}-\frac{1}{3}a_1\{a_1,b_2,c_2\}, \\
\{b_2,c_2,d_3\} & = \{b_2,c_2,g_3\}+\frac{1}{3}b_2\{a_1,b_2,c_2\}-\frac{1}{3}c_2\{a_1,b_2,c_2\},
\end{align*}
and $W((\{*,*,g_3\}|[M_2])_{11})=0$, so the condition for 
\eqref{shiki3} is given by 
\begin{equation}\label{eq2}
F_1-\frac{1}{3}a_1F_2-\frac{1}{3}a_1F_3 +\frac{1}{3}(b_2-c_2)F_4
\underset{M_2}{\sim} -A^2B^4C^4a^2b^4c^4,
\end{equation}
where $\underset{M}{\sim}$ means that the images of $W(*|[M])$ of both sides coincide.
Comparing both sides for  \eqref{shiki2} and \eqref{shiki4},  we have 
\[
F_1\underset{M_1}{\sim} -A^2B^4C^4a^2b^4c^4,
\]
and 
\[
F_4\underset{M_3}{\sim} -A^2B^2C^2a^2b^2c^2.
\]
Since we have $W(d_3|[M_3])=-A^2B^2C^2a^2b^2c^2$, we have 
$F_4-d_3\underset{M_3}{\sim} 0$, This means that 
$F_4-d_3=a_1G$ for some $G\in A_2^{I}(\Gamma_0^0(2)^{\psi})$. 
So we have 
\[
F_4\{b_2,c_2,d_3\}=d_3\{b_2,c_2,d_3\}+a_1G\{b_2,c_2,d_3\}.
\]  
But $a_1\{b_2,c_2,d_3\}$ is in $Span(\{a_1,b_2,c_2\},\{a_1,b_2,d_3\},\{a_1,c_2,d_3\})$,
so adjusting $F_1$, $F_2$, $F_3$, we may assume that 
$F_4=d_3$.
If we put 
\[
F_0=-\frac{1}{9}a_1(-6a_1^2b_2+6a_1^2c_2+2b_2^2-b_2c_2-c_2^2),
\]
then we hve $W(F_0|[M_1])=-A^2B^4C^4a^2b^4c^4$.
This means that $F_1=F_0+d_3G_2$ for some $G_2$. 
Now the simultaneous equation \eqref{eq1} and \eqref{eq2} 
is now written as 
\begin{align*}
\frac{a_1}{3}(F_2-2F_3)& \sim -A^2B^4C^4a^2b^4c^4-F_1-\frac{1}{3}(c_2+2b_2-18a_1^2)d_3, 
\\
-\frac{a_1}{3}(F_3+F_2)& \underset{M_2}{\sim} -A^2B^4C^4a^2b^4c^4-F_1-\frac{1}{3}(b_2-c_2)d_3.
 \end{align*}
We can show this has a solution for $F_1=F_0$. Indeed the following 
forms give a solution:
\begin{align*}
F_1& =F_0, \\
F_2 & =  4a_1^4-\frac{8b_2^2}{9}+2a_1^2c_2
-\frac{14}{9}b_2c_2+\frac{4c_2^2}{9}+18a_1d_3,\\
F_3 & = -4a_1^4+\frac{2}{9}b_2^2-2a_1^2c_2
+\frac{8b_2c_2}{9}+\frac{8c_2^2}{9}-6a_1d_3,\\
F_4 & = d_3,
\end{align*}
So we can show that $(\X_{11},\frac{(2\pi i)^2}{11\cdot 16}H)$
is in the image of $\xi$ for the  $H$ of \eqref{H}, where $F_i$ are given as above.

Now we will determine $h\in A_{*,2}^{II}(\Gamma_0^0(2)^{\psi})$ such that $(0,h)$ is in the image of $\xi$.
We put $M_0=1_4$ and write  
\[
V_i=\{h\in A_{*,2}^{II}(\Gamma_0^0(2)^{\psi}); W((h|[M_i])_{11})=0\}.
\]
Here $A_{*,2}^{II}(\Gamma_0^0(2)^{\psi})$ is 
spanned by $\{a_1,b_2,c_2\}$, $\{a_1,b_2,d_3\}$, $\{a_1,c_2,d_3\}$, $\{b_2,c_2,d_3\}$.
The fundamental relation over $A^{I}(\Gamma_0^0(2)^{\psi})$ is 
\begin{equation}\label{fund00}
a_1\{b_2,c_2,d_3\}+b_2\{a_1,c_2,d_3\}+c_2\{a_1,b_2,d_3\}+d_3\{a_1,b_2,c_2\}=0.
\end{equation}
Here we may replace $d_3$ by $f_3$ or $g_3$. 
Since we have $W(f_3)=W(\p_{ii}f)=0$, $W(d_3|[M_1])=W(\p_{12}(d_3|[M_1])=0$, 
$W(g_3|[M_2])=W(\p_{12}(g_3|[M_2]))=0$, $W(a_1|[M_3])=W(\p_{12}(a_1|[M_3])=0$,
and since we see that $W_{11}(\{a_1,b_2,c_2\}_{11})$, $W_{11}(\{a_1,b_2,c_2\}|[M_1])$,
$W_{11}(\{a_1,b_2,c_2\}|[M_2])$, $W_{11}(\{b_2,c_2,d_3\})$ are all non-zero,  
we have 
\begin{align*}
V_0 & =Span(f_3\{a_1,b_2,c_2\},\{a_1,b_2,f_3\},\{a_1,c_2,f_3\},\{b_2,c_2,f_3\}), \\
V_1 & = Span(d_3\{a_1,b_2,c_2\},\{a_1,b_2,d_3\},\{a_1,c_2,d_3\}.\{b_2,c_2,d_3\}), \\
V_2 & = Span(g_3\{a_1,b_2,c_2\},\{a_1,b_2,g_3\},\{a_1,c_2,g_3\},\{b_2,c_2,g_3\}), \\
V_3 & = Span(a_1\{b_2,c_2,d_3\},\{a_1,b_2,c_2\},\{a_1,b_2,d_3\},\{a_1,c_2,d_3\}).
\end{align*}
By \eqref{fund00} or by its variant, we also have
\begin{align*}
V_0& = Span(\{a_1,b_2,f_3\},\{a_1,c_3,f_3\},\{b_2,c_2,f_3\}),\\
V_1& = Span(\{a_1,b_2,d_3\},\{a_1,c_2,d_3\},\{b_2,c_2,d_3\}), \\
V_2& = Span(\{a_1,b_2,g_3\},\{a_1,c_2,g_3\},\{b_2,c_2,g_3\}), \\
V_3 & = Span(\{a_1,b_2,c_2\},\{a_1,b_2,d_3\},\{a_1,c_2,d_3\}).
\end{align*}
We determine $V_1\cap V_3$. We write an element $h_1\in V_1\cap V_3$ as
\begin{align*}
h_1& =F_1\{a_1,b_2,d_3\}+F_2\{a_1,c_2,d_3\}+F_3\{b_2,c_2,d_3\}
\\ & =F_4\{a_1,b_2,c_2\}+F_5\{a_1,b_2,d_3\}+F_6\{a_1,c_2,d_3\}.
\end{align*}
Since the fundamental relation is \eqref{fund00}, we have 
\[
F_1-F_5=c_2G, \quad F_2-F_6=b_2G, \quad F_3=a_1G, \quad F_4=d_3G.
\]
This means that 
\[
h_1=G d_3\{a_1,b_2,c_2\}+F_5\{a_1,b_2,d_3\}+F_6\{a_1,c_2,d_3\}
\]
for arbitrary $G$, $F_5$, $F_6$ so 
\[
V_1\cap V_3=Span(d_3\{a_1,b_2,c_2\},\{a_1,b_2,d_3\},\{a_1,c_2,d_3\}).
\]
Since $d_3=f_3-\frac{1}{2}a_1(a_1^2-3b_2-6c_2)$, we can replace $d_3$ by $f_3$ in 
the generator of $V_3$. So in the same way as before, we see 
\[
V_0\cap V_3=Span(f_3\{a_1,b_2,c_2\},\{a_1,b_2,f_3\},\{a_1,c_2,f_3\}).
\]
Since $d_3=g_3+\frac{1}{3}a_1b_2-\frac{1}{3}a_1c_2$, in the same way we have 
\[
V_2\cap V_3 = Span(g_3\{a_1,b_2,c_2\},\{a_1,b_2,g_3\},\{a_1,c_2,g_3\}).
\]
Next for a fixed $h_2\in V_1\cap V_3$, we write
\[
h_2=F_1d_3\{a_1,b_2,c_2\}+F_2\{a_1,b_2,d_3\}+F_3\{a_1,c_2,d_3\}. 
\]
Then we have 
\[
h_2=(F_1d_3+\frac{1}{3}a_1F_2-\frac{2}{3}a_1F_3)\{a_1,b_2,c_2\}+F_2\{a_1,b_2,f_3\}
+F_3\{a_1,c_2,f_3\},
\]
and if $h_2\in V_0\cap V_3$ besides, then 
this means that $F_1d_3+\frac{a_1}{3}(F_2-F_3)=f_3G$ for some $G$. 
This is equivalent to say that 
\begin{equation}\label{eq3}
F_1(-2a_1^2+\frac{2}{c}b_2+\frac{1}{3}c_2)+\frac{1}{3}(F_2-2F_3) 
= f_3G_1
\end{equation}
for some $G_1$. In the same way, if we assume that $h_2\in V_2\cap V_3$, then 
we have 
\begin{equation}\label{eq4}
\frac{1}{3}(b_2-c_2)F_1-\frac{1}{3}F_2-\frac{1}{3}F_3=g_3G_2
\end{equation}
for some $G_2$. Regarding \eqref{eq3} and \eqref{eq4} as a simultaneous 
equation with respect to $F_2$ and $F_3$, we have 
\begin{align*}
F_2 & = f_3G_1-2g_3G_2-(-2a_1^2+c_2)F_1, \\
F_3 & =-f_3G_1-g_3G_2+(-2a_1^2+b_2)F_1. \\
\end{align*}
Writing $h_2$ by using this, we have 
\begin{align*}
h_2= & F_1(d_3\{a_1,b_2,c_2\}+(2a_1^2-c_2)\{a_1,b_2,d_3\}+(-2a_1^2+b_2)\{a_1,c_2,d_3\})
\\ & +G_1f_3(\{a_1,b_2,d_3\}-\{a_1,c_2,d_3\})
\\ & - G_2g_3(2\{a_1,b_2,d_3\}+\{a_1,c_2,d_3\}).
\end{align*}
This means the following theorem.
\begin{theorem}
$J_{*,1}^{II}(\Gamma_0^0(2)^{\psi})$ is spanned over $A^I(\Gamma_0^0(2)^{\psi})$ 
by the following $4$ linearly independent basis of weight 
$11$, $9$, $10$, $10$, whose images by $\xi$ are
\begin{align*}
& (\X_{11},\frac{(2\pi i)^2}{11\cdot 16}H), \\
& (0, d_3\{a_1,b_2,c_2\}+(2a_1^2-c_2)\{a_1,b_2,d_3\}+(-2a_1^2+b_2)\{a_1,c_2,d_3\}), 
\\ & (0,f_3(\{a_1,b_2,d_3\}-\{a_1,c_2,d_3\})), \\
& (0,g_3(2\{a_1,b_2,d_3\}+\{a_1,c_2,d_3\})).
\end{align*}
\end{theorem}

\end{document}